\documentclass{article}
\usepackage[utf8]{inputenc}

\usepackage{amsfonts,amsmath,amssymb,amsthm}
\usepackage{enumerate}
\usepackage[T1]{fontenc} 
\usepackage{bbm} 
\usepackage{verbatim}
\usepackage{authblk}
\usepackage{tikz}
\usepackage{soul}

\usetikzlibrary{calc,patterns}
\usetikzlibrary{arrows,shapes,positioning}
\usetikzlibrary{decorations.markings}
\tikzstyle arrowstyle=[scale=1]
\tikzstyle directed=[postaction={decorate,decoration={markings,mark=at position .6 with {\arrow[arrowstyle]{stealth}}}}]

\DeclareMathOperator{\dens}{dens}
\DeclareMathOperator{\codens}{codens}
\DeclareMathOperator{\Var}{Var}
\DeclareMathOperator{\im}{Im}
\newcommand{\vide}{\textup{\O}}
\newcommand{\moinsun}{^{-1}}
\newcommand{\bs}{\boldsymbol}
\newcommand{\tendvers}{\xrightarrow[n\to\infty]{}}
\def\Pr{\mathbf{Pr}}
\newcommand\PrCond[2]{\Pr\left( #1 \;\middle|\; #2 \right)}
\newcommand\Cond[2]{\left( #1 \;\middle|\; #2 \right)}
\newcommand\flr[1]{\lfloor #1 \rfloor}
\newcommand\esp[1]{\mathbb{E}\left( #1 \right)}

\newtheorem{thm}{Theorem}[section]
\newtheorem{thmm}{Theorem}

\newtheorem{defi}[thm]{Definition}
\newtheorem{lem}[thm]{Lemma}
\newtheorem{prop}[thm]{Proposition}
\newtheorem{cor}[thm]{Corollary}
\newtheorem*{nota}{Notation}
\newtheorem{manualtheoreminner}{Theorem}
\newenvironment{thmp}[1]{%
  \renewcommand\themanualtheoreminner{#1}%
  \manualtheoreminner
}{\endmanualtheoreminner}

\theoremstyle{remark}

\newtheorem{ex}[thm]{Example}

\title{Density of random subsets and applications to group theory}
\author{\textsc{Tsung-Hsuan Tsai}}
\affil{Institut de Recherche Mathématique Avancée\\
\small{\quad\\ \textit{Address:} UFR de Mathématiques (Office 105)\\ 7 rue René Descartes
67000 Strasbourg (France).\\
\textit{E-mail:} tsung-hsuan.tsai@math.unistra.fr}}
\affil{MSC2020: 20F05, 20F06, 60C05}
\affil{Key words: random group, intersection formula, cancellation theory}
\date{}

\begin{document}
\maketitle
\begin{abstract} 
    Developing an idea of M. Gromov in \cite{Gro93} 9.A, we study the intersection formula for random subsets with density. The \textit{density} of a subset $A$ in a finite set $E$ is defined by $\dens A := \log_{|E|}(|A|)$. The aim of this article is to give a precise meaning of Gromov's \textit{intersection formula}: "Random subsets" $A$ and $B$ of a finite set $E$ satisfy $\dens (A\cap B) = \dens A + \dens B -1$.\\
    
    As an application, we exhibit a phase transition phenomenon for random presentations of groups at density $\lambda/2$ for any $0<\lambda<1$, characterizing the $C'(\lambda)$-small cancellation condition. We also improve an important result of random groups by G. Arzhantseva and A. Ol'shanskii in \cite{AO96} from density $0$ to density $0\leq d<\frac{1}{120m^2\ln(2m)}$.
\end{abstract}
\tableofcontents

\section*{Introduction}
\paragraph{Density of subsets.}

Let $A$ be a subset of a finite set $E$. Denote $|E|$, $|A|$ their cardinalities. In \cite{Gro93} 9.A, M. Gromov defined the \textit{density} of $A$ in $E$ as
\[\dens_E (A) := \log_{|E|}(|A|).\]

Namely, $\dens_E(A)$ is the number $d \in \{-\infty\}\cup [0,1]$ such that $|A|=|E|^{d}$. Note that $d = -\infty$ if and only if $A=\vide$. If the set $E$ is fixed, we omit the subscript and simply denote the density by $\dens A$.

In \cite{Gro93} p.270, the \textit{intersection formula} is stated as follows: \textit{Random subsets $A$ and $B$ of a finite set $E$ satisfy
\[\dens (A\cap B) = \dens A + \dens B -1\]
with the convention}
\[\dens A < 0 \iff A = \vide.\]
\quad

If $E$ is a finite-dimensional vector space over a finite field, every affine subspace $A$ satisfies $\dens A = \dim A/\dim E$. The intersection formula is then a "random subset version" of the well-known result for affine subspaces: \textit{Transversal affine subspaces $A$ and $B$ of a vector space $E$ satisfy
\[\dim (A\cap B) = \dim A + \dim B -\dim E\]
with the convention}
\[\dim A < 0 \iff A = \vide.\]

\paragraph{Purpose of the paper.}  In \cite{Gro93} p.270 "explanation", Gromov did not give a precise definition of a random subset with density. In \cite{Gro93} p.272, he proposed that one can consider the class of random subsets defined by measures \textit{invariant under the permutations of $E$}.

In this article, we discuss two basic models of random subsets that are contained in the permutation invariant model: The \textit{uniform density model} and the \textit{Bernoulli density model}. The first one is defined by the uniform distribution on all subsets of $E$ with cardinality $\flr{|E|^d}$. This model is used by Y. Ollivier in \cite{Oll03}, \cite{Oll04} and \cite{Oll07} to study the density model of random groups, and by A. \.Zuk in \cite{Zuk03} to construct random triangular groups. For the Bernoulli density model, every element in $E$ is taken independently with the same probability $|E|^{d-1}$. This model is considered by Antoniuk-Łuczak-\'Swi\k atkowski in \cite{ALS15} to study random triangular groups.

The aim of this article is to establish a general framework for the study of random subsets with densities, and to proof the intersection formula for the class of random subsets that are \textit{densable} and \textit{permutation invariant}.

\paragraph{Random subsets and the intersection formula.}

In the first section, we introduce the notion of \textit{densable sequences of random subsets}. Let $E$ be a finite set. A \textit{random subset} of $E$ is a $\mathcal{P}(E)$-valued random variable, where $\mathcal{P}(E)$ is the set of subsets of $E$. Note that $|A|$ is a usual real-valued random variable. The \textit{density} of $A$ in $E$, defined by $\dens_E A := \log_{|E|} (|A|)$, is hence a random variable with values in $\{-\infty\}\cup[0,1]$.

As our approach is asymptotic when $|E|\to\infty$, we consider a sequence of finite sets $\bs E = (E_n)_{n\in\mathbb{N}}$ where $|E_n|\to\infty$. A \textit{sequence of random subsets} of $\bs E$ is a sequence $\bs A = (A_n)$ where $A_n$ is a random subset of $E_n$ for all $n\in \mathbb{N}$. Such a sequence is \textit{densable} with density $d\in \{-\infty\}\cup[0,1]$ if the sequence of random variables $\dens_{E_n}(A_n)$ converges weakly (i.e. converges in distribution) to the constant $d$ (c.f. \cite{Gro93} p.272). For a sequence of properties $\bs Q = (Q_n)$, we say that $Q_n$ is \textit{asymptotically almost surely} (a.a.s.) satisfied if the probability that $Q_n$ is satisfied goes to $1$ when $n\to\infty$. For example, for a sequence of random subsets $\bs A = (A_n)$, $\dens \bs A = -\infty$ if and only if a.a.s. $A_n=\vide$.\\

In Section 2 we work on the permutation invariant model in \cite{Gro93} p.272. Let $E$ be a finite set. A random subset $A$ of $E$ is \textit{permutation invariant} if its law is invariant under the permutations of $E$. Namely, for any subset $a\in \mathcal{P}(E)$ and any permutation $\sigma \in \mathcal{S}(E)$, we have $\Pr(A = a) = \Pr(A = \sigma(a))$.

Consider a sequence of finite sets $\bs E = (E_n)$ with $|E_n|\to\infty$. Denote by $\mathcal{D}(\bs E)$ the class of \textit{densable sequences of permutation invariant random subsets} of $\bs E$. We prove the intersection formula stated as follows: 

\begin{thmm}[The intersection formula, Theorem \ref{intersection}] Let $\bs A = (A_n), \bs B = (B_n)$ be independent sequences of random subsets in $\mathcal{D}(\bs E)$ with densities $\alpha,\beta$. If $\alpha + \beta \neq 1$, then the sequence of random subsets $\bs A\cap \bs B$ is also in $\mathcal{D}(\bs E)$. In addition: 
    \[\dens (\bs A\cap \bs B) = \begin{cases}
        \alpha + \beta - 1 & \textup{ if } \alpha + \beta > 1\\
        -\infty & \textup{ if } \alpha + \beta < 1.
        \end{cases}\]
\end{thmm}

The density $-\infty$ means that a.a.s. the random subset is empty.\\

In Section 3, we study the intersection between a random subset and a \textit{fixed} subset. We develop a generalized form: the \textit{multi-dimensional intersection formula}. Let $\bs E = (E_n)$ be a sequence of finite sets with $|E_n|\to\infty$. Denote $E_n^{(k)}$ the set of pairwise distinct $k$-tuples of the set $E_n$. Let $\bs A$ be a sequence of random subsets in $\mathcal{D}(\bs E)$ (densable and permutation invariant). We are interested in the intersection between $\bs A^{(k)}$ and a densable sequence of subsets $\bs X$ of $\bs E^{(k)}$.

For $k\geq 2$, the intersection formula is in general not correct (see example \ref{cex no intersection formula}). We show that by an additional \textit{self-intersection condition} on $\bs X$, we can achieve an intersection formula.

\begin{thmm}[The multi-dimensional intersection formula, Theorem \ref{multidim intersection}] Let $\bs A = (A_n)$ be a sequence of random subsets in $\mathcal{D}(\bs E)$ with density $0<d<1$. Let $\bs X = (X_n)$ be a densable sequence of fixed subsets of $\bs E^{(k)}$ with density $\alpha$.
    \begin{enumerate}[(i)]
        \item If $d + \alpha <1$, then a.a.s.
        \[A_n^{(k)}\cap X_n = \vide.\]
        \item If $d + \alpha >1$ and $\bs X$ satisfies the $d$-small self intersection condition (Definition \ref{multidim condition}), then the sequence of random subsets $\bs A^{(k)}\cap \bs X$ is densable and 
        \[\dens (\bs A^{(k)}\cap \bs X) = \alpha + d - 1.\]
    \end{enumerate} 
\end{thmm}

The intersection formula in $\bs E$ between a random subset and a fixed subset is a special case of this theorem by taking $k=1$.

\paragraph{Applications to group theory: random groups.}

The last section is dedicated to applications to group theory, more precisely to small cancellation theory.\\

The first mention of generic property for finitely presented groups appears in the late 80's in the works of V. S. Guba \cite{Guba86} and M. Gromov \cite{Gro87}. In \cite{Guba86}, the author showed that for "almost every" group presented by $m\geq4$ generators and one "long" relator, any $2$-generated subgroup is free. In \cite{Gro87}, Gromov defined two models of random group presentations with fixed number of generators and relators.

In 1993, Gromov introduced the \textit{density model} of random groups in \cite{Gro93} 9.B. The number of generators is still fixed, but the number of relators \textit{grows exponentially} with the length of the relators, determined by a density parameter $d$. A \textit{phase transition phenomenon} is then stated as follows: if $d < 1/2$, then a.a.s. the random group is infinite hyperbolic; whereas if $d>1/2$, then a.a.s. the random group is trivial.

In a 1996 paper \cite{AO96}, G. Arzhantseva and A. Ol’shanskii generalized Guba's result. They proved that for "almost every" group presented by $m\geq2$ generators and $k\geq1$ long relators , any $(m-1)$-generated subgroup is free. In their model, the number of generators $k$ is fixed, as in Gromov's 1987 model \cite{Gro87}. This model is called the Arzhantseva-Ol’shanskii model, or the \textit{few relator model} of random groups.

For more detailed surveys on random groups, see (in chronological order) \cite{Ghys04} by E. Ghys, \cite{Oll05} by Y. Ollivier, \cite{KS08} by I. Kapovich and P. Schupp and \cite{BNW} by F. Bassino, C. Nicaud and P. Weil.

Fix a set of alphabets $X = \{x_1,\dots, x_m\}$ as generators of groups. Denote by $B_\ell$ the set of cyclically reduced words of $X^\pm$ of lengths at most $\ell$. If $S_\ell$ is the set of cyclically reduced words of length exactly $\ell$, it is clear that $2m(2m-1)^{\ell-2}(2m-2) \leq |S_\ell| \leq 2m(2m-1)^{\ell-1}$. So

\[\frac{2m}{2m-1}\left[(2m-1)^\ell-1\right] \leq |B_\ell| \leq \frac{2m}{2m-2}\left[(2m-1)^\ell-1\right].\]

As we are interested in asymptotic behaviors when $\ell\to\infty$, we can write $|B_\ell| = (2m-1)^{\ell+O(1)}$. Consider $\bs B = (B_\ell)_{\ell\geq 1}$ as our ambient sequence of sets. Let $d\in \{-\infty\} \cup[0,1]$. A \textit{sequence of random groups} with density $d$, denoted by $\bs G(m,d) = (G_\ell(m,d))$, is defined by random presentations $G_\ell(m,d) = \langle X | R_\ell \rangle$ where $\bs R = (R_\ell)$ is a densable sequence of permutation invariant random subsets of $\bs B$ with density $d$.\\

The first mention of the $\lambda/2$ phase transition for the $C'(\lambda)$-small cancellation condition is by Gromov in \cite{Gro93} p.273, showing that if $2d<\lambda$ then a random group at density $d$ satisfies $C'(\lambda)$. He remarked also that, in particular, if $d<1/12$ then the group is hyperbolic; and if $d>1/12$ then the group is not $C'(1/6)$. Ollivier-Wise gave a detailed proof of $d<\lambda/2$ implying $C'(\lambda)$ in \cite{OW11} Proposition 1.8. In \cite{Oll05} p.31 Ollivier stated the phase transition : if $d>\lambda/2$ then $C'(\lambda)$ does not hold. However, his "dimension reasoning" is the $2$-dimensional intersection formula between a random subset (pairs of distinct relators in a random group) and a \textit{fixed} subset (pairs of distinct relators denying $C'(\lambda)$), which does not hold in general (as Example 3.3 shows). 

Bassino-Nicaud-Weil gave a proof of $d>\lambda/2$ implying non-$C'(\lambda)$ in \cite{BNW} p.7 (Theorem 2.1). Their argument showed that the pairs of distinct relators \textit{in a random group} denying $C'(\lambda)$ is not empty, but did not give its density.

The $d$-small self-intersection condition (Definition 3.6) for a fixed subset is introduced to rule out this difficulty. By the multi-dimensional intersection formula (\textbf{Theorem 2}), we show that if $d>\lambda/2$, then the pairs of distinct relators in a random group denying $C'(\lambda)$ is with density $d-\lambda/2$ and hence not empty.

\begin{thmm}[Phase transition at density $\lambda/2$, Theorem \ref{lambda small cancellation}] Let $\bs G(m,d) = (G_\ell(m,d))$ be a random group with $m$ generators and with density $d$. Let $\lambda\in ]0,1[$.
\begin{enumerate}
    \item If $d<\lambda/2$, then a.a.s. $G_\ell(m,d)$ satisfies $C'(\lambda)$.
    \item If $d>\lambda/2$, then a.a.s. $G_\ell(m,d)$ does not satisfy $C'(\lambda)$.
\end{enumerate}
\end{thmm}\quad

It was given as an "interesting problem" in \cite{Oll05} I.3.c that some algebraic properties of random groups at density $0$ (\cite{AO96} by Arzhantseva-Ol'shanskii, \cite{Arz97}, \cite{Arz98}, \cite{Arz00} by Arzhantseva, and \cite{KS05} by Kapovich-Schupp) may extend to some positive density $d$. In \cite{KS08} Theorem 7.5, Kapovich and Schupp extends Arzhantzeva's "all $L$-generated subgroups of infinite index are free" result \cite{Arz97} (for a fixed $L>0$) to some density $d>0$ independent of $m$. A property is called "low-density random" by Kapovich-Schupp in \cite{KS08} p.3 if the corresponding density $d(m)$ is positive but converges to $0$ when $m$ goes to infinity. They claimed that Arzhantseva-Ol'shanskii's "all $(m-1)$-generated subgroups are free" result \cite{AO96} is a low-density random property (\cite{KS08} Theorem 1.1 (2), Theorem 5.4 (2)), but the density $d(m)$ is not given.

In our study, the number of generators $m$ is fixed, and we look for a density $d(m)$ such that the "all $(m-1)$-generated subgroups are free" property holds for a random group with $m$ generators of density $d<d(m)$. Using \textbf{Theorem2} and \textbf{Theorem 3}, we give an explicit bound $d(m)=\frac{1}{120m^2\ln(2m)}$ that extends Arzhantseva-Ol'shanskii's result in \cite{AO96} from density $0$ to density $0 \leq d < d(m)$.

\begin{thmm}[Every $(m-1)$-generated subgroup is free, Theorem \ref{AO with density}]
Let $(G_\ell(m,d))$ be a sequence of random groups with $m$ generators and with density \[0\leq d<\frac{1}{120m^2\ln(2m)}.\] 
Then a.a.s. every $(m-1)$-generated subgroup of $G_\ell(m,d)$ is free.
\end{thmm}

Ollivier remarked in \cite{Oll05} p.71 that at density $d>1-\log_{2m-1}(2m-3)$, the rank of a random group with $m$ generators with density $d$ is at most $m-1$, so the "all $(m-1)$-generated subgroups are free" property fails. There is still a large gap between $\log_{2m-1}(2m-3) \sim \frac{1}{m\ln(2m)}$ and $\frac{1}{120m^2\ln(2m)}$.\\[1em]

\paragraph{Acknowledgements.} I would like to thank my supervisor, Thomas Delzant, for his guidance and interesting discussions on the subject, especially for his patience with me while completing this article. I would also like to thank the referee for his/her thorough review of the manuscript and highly appreciate the comments and suggestions, which significantly contributed to improving the quality of this work.

\section{Definitions and basic models}

\subsection{Densable sequences of random subsets}

Let $E$ be a finite set, denote $|E|$ its cardinality. The following definition is due to M. Gromov in \cite{Gro93}.

\begin{defi} Let $E$ be a finite non-empty set and $A\subset E$. The density of $A$ in $E$ is defined by
\[\dens_E A := \log_{|E|}|A| = \frac{\log |A|}{\log |E|}.\]
So that $d\in[0,1]\cup\{-\infty\}$ is a real number such that $|E|^d = |A|$.
\end{defi}

We will omit the subscript $E$ if the set is fixed and simply denote the density by $\dens A$. Note that $\dens A = -\infty$ if and only if $A = \vide$.\\
\begin{defi}
Let $E$ be a finite set. Denote $\mathcal{P}(E)$ the set of subsets of $E$. A random subset $A$ of $E$ is a $\mathcal{P}(E)$-valued random variable.
\end{defi}

In this article, we use upper-case letters $A,B,C,\dots$ to denote random subsets and lower-case letters $a,b,c,\dots$ to denote fixed subsets. The law of a random subset $A$ is determined by instances $\Pr(A = a)$ through all subsets $a\in \mathcal{P}(E)$ (or $a\subset E$). Its cardinality $|A|$ is a usual real-valued random variable.\\

Here we give some basic examples of random subsets.
\begin{ex}(Examples of random subsets)\label{ex random subsets}
\begin{enumerate}[(i)]
    \item (Dirac model) A fixed subset $c\subset E$ can be regarded as a constant random subset. Its law is
    \[\Pr(A=a)=\begin{cases}
        1 & \textup{ if } a = c \\
        0 & \textup{ if } a\neq c.
        \end{cases}\]

    \item (Uniform random subset) Fix an integer $k\leq |E|$. Let $A$ be the uniform distribution on all subsets of $E$ of cardinality $k$. Its law is
    \[\Pr(A=a)=\begin{cases}
        \binom{|E|}{k}\moinsun & \textup{ if } |a| = k \\
        0 & \textup{ if } |a| \neq k.
        \end{cases}\]
    \item (Bernoulli random subset) Let $A$ be the Bernoulli sampling of parameter $p\in[0,1]$ on the set $E$: The events $\{x\in A\}$ through all $x\in E$ are independent of the same probability $p$. The law of $A$ is
    \[\Pr(A=a)=p^{|a|}(1-p)^{|E|-|a|}.\]
    In this case $|A|$ follows the binomial law $\mathcal{B}(|E|,p)$.
\end{enumerate}
\end{ex}\quad

As usual random variables, a random subset can be constructed by other random subsets.
\begin{ex}(Set theoretic operations)
The intersection of two random subsets $A,B$ of a finite set $E$ is another random subset. The law of $A\cap B$ is
\[\Pr(A\cap B = c) = \sum_{a,b\in\mathcal{P}(E);a\cap b = c}\Pr(A = a , B = b).\]

In particular, if $A,B$ are \textit{independent} random subsets, then 
\[\Pr(A\cap B = c) = \sum_{a,b\in\mathcal{P}(E);a\cap b = c}\Pr(A = a)\Pr(B = b).\]

The union of two subsets and the complement of a subset are similarly defined.
\end{ex}\quad\\

We are interested in the asymptotic behavior of random subsets when $|E|\to\infty$. Consider a sequence of finite sets $\bs{E} = (E_n)_{n\in\mathbb{N}}$ with $|E_n|\xrightarrow[n\to\infty]{}\infty$. Recall that the density of a subset $a\subset E$ is defined by $\dens_E (a) := \log_{|E|}|a|$.

\begin{defi}[Densable sequence of random subsets]\quad
\begin{enumerate}[$(i)$]
    \item A sequence of (fixed) subsets of $\bs E = (E_n)$ is a sequence $\bs{a} = (a_n)$ such that $a_n\subset E_n$ for all $n$.
    
    A sequence of subsets $\bs{a}$ is \textbf{densable} with density $d \in [0,1]\cup \{-\infty\}$ if 
    \[\dens_{E_n}(a_n) = \log_{|E_n|}|a_n|\tendvers d.\]
    \item Similarly, a sequence of random subsets of $\bs E$ is a sequence $\bs{A} = (A_n)$ such that $A_n$ is a random subset of $E_n$ for all $n$.
    
    A sequence of random subsets $\bs{A}$ is \textbf{densable} with density $d\in [0,1]\cup \{-\infty\}$ if the sequence of real-valued random valuables $\dens_{E_n}(A_n) = \log_{|E_n|}|A_n|$ converges in distribution to the constant $d$. 

    \item Two sequences of random subsets $\bs A = (A_n)$, $\bs B = (B_n)$ of $\bs E$ are independent if $A_n$, $B_n$ are independent random subsets of $E_n$ for all $n$.
\end{enumerate}
\end{defi}\quad\\

Here we give some examples of sequences of random subsets.
\begin{ex}[Examples of densable sequences of random subsets]\label{ex densable sequences}\quad
\begin{enumerate}[(i)]
    \item For a fixed sequence of subsets $\bs a = (a_n)$, $\dens(\bs a) = -\infty$ if and only if $a_n = \vide$ for large enough $n$.

    \item A densable sequence of subsets $\bs{a} = (a_n)$ can be regarded as a densable sequence of random subsets (Dirac model on each term). If we take $|a_n| = \lfloor|E_n|^d\rfloor$ with some $0\leq d\leq 1$, then $\bs a$ is densable with density $d$.
    
    \item (Uniform density model) Let $\bs A = (A_n)$ be a sequence of random subsets of $\bs E$. $\bs A$ is a sequence of uniform random subsets with density $d$ if $A_n$ is the uniform distribution on all subsets of $E_n$ of cardinality $\lfloor|E_n|^d\rfloor$. Its law is
        \[\Pr(A_n=a)=\begin{cases}
        \binom{|E_n|}{\lfloor|E_n|^d\rfloor}\moinsun & \textup{ if } |a| = \lfloor|E_n|^d\rfloor \\
        0 & \textup{ if } |a| \neq \lfloor|E_n|^d\rfloor.
        \end{cases}\]
    \item (Bernoulli density model) Let $d>0$. If $A_n$ is a Bernoulli sampling of $E_n$ with parameter $|E_n|^{d-1}$, then $\bs A = (A_n)$ is a sequence of densable random subsets of $\bs E$. It is rather not obvious that such sequences are densable (see Proposition \ref{densBernoulli}).
\end{enumerate}
\end{ex}\quad

\begin{defi}\label{definition a.a.s.}
Let $\bs{Q} = (Q_n)$ be a sequence of events. The event $Q_n$ is \textbf{asymptotically almost surely} true if $\Pr(Q_n)\xrightarrow[n\to\infty]{}1$. 

Equivalently, for any $p<1$ arbitrary close to $1$ we have $\Pr(Q_n)>p$ for $n$ large enough. We denote briefly a.a.s. $Q_n$.
\end{defi}

For example, if $\bs A$ is a sequence of random subsets with $\dens \bs A = -\infty$, then $\Pr(|A_n| = 0)\tendvers 1$. Which is equivalent to a.a.s. $|A_n| = 0$, or a.a.s. $A_n = \vide$.\\

\begin{prop}[Characterization of densability]\label{Concentration}
Let $\bs A$ be a sequence of random subsets of $\bs E$. Let $d\geq 0$. $\bs A$ is densable with density $d$ if and only if
\[\forall \varepsilon > 0 \textup{ a.a.s. } |E_n|^{d-\varepsilon}\leq |A_n|\leq |E_n|^{d+\varepsilon}. \]
\end{prop}
\begin{proof}
The convergence in distribution to a constant is equivalent to the convergence in probability. So $\log_{|E_n|}|A_n|$ converges in distribution to $d$ if and only if
\[\forall \varepsilon > 0 \quad \Pr(|\log_{|E_n|}|A_n|-d|\leq \varepsilon)\tendvers 1,\]
which gives the estimation
\[\forall \varepsilon > 0 \textup{ a.a.s. } |E_n|^{d-\varepsilon}\leq |A_n|\leq |E_n|^{d+\varepsilon}.\]
\end{proof}

In general, the intersection of two densable sequences is not necessarily densable. The intersection formula is then \textit{not} satisfied by the class of densable sequences of random subsets. Here is a simple example.

\begin{ex} Let $\bs E = (E_n)$ be a sequence of sets with $|E_n| = 2n$. Let $\bs a = (a_n)$, $\bs b = (b_n)$ be sequences of subsets of $\bs E$ such that $b_n = E_n\backslash a_n$ and $|a_n|=|b_n|=n$. They are both densable subsets with density $1$ because $\log (n)/\log (2n)\to 1$. Whereas $\dens (\bs a\cap \bs b) = -\infty$. They do not verify the intersection formula.

Define another sequence of subset $\bs c = (c_n)$ by $c_n := a_n$ if $n$ is odd and $c_n := b_n$ if $n$ is even. By its definition, $\bs c$ is densable with density $1$. But the intersection $\bs b\cap \bs c$ is empty when $n$ is odd and non-empty when $n$ is even, so $\bs b\cap \bs c$ is not densable.
\end{ex}\quad

\subsection{The Bernoulli density model}

Let $\bs E = (E_n)$ with $|E_n|\to\infty$ be the ambient sequence of sets.

\begin{defi}[Bernoulli density model]
Let $d\leq 1$. Let $\bs{A} = (A_n)$  be a sequence of random subsets of $\bs E$. It is a sequence of Bernoulli random subsets with density $d$ if $A_n$ is a Bernoulli sampling of $E_n$ with parameter $|E_n|^{d-1}$.
\end{defi}

This model is particularly easy to manipulate. We will see that it is densable, closed under intersection and verifies the intersection formula.

Recall that the real-valued random variable $|A_n|$ follows the binomial law $\mathcal{B}(|E_n|,|E_n|^{d-1})$. Thus $\mathbb{E}(|A_n|) = |E_n|^d$.

\begin{lem}[Concentration lemma]\label{concentration Bernoulli} Let $\bs{A}$ be a sequence of Bernoulli random subsets with density $d>0$. Then a.a.s.
\[\left||A_n| - |E_n|^d\right|\leq \frac{1}{2}|E_n|^{d}.\]
\end{lem}
\begin{proof}
By Chebyshev's inequality,
\[\Pr\left(\left||A_n|-|E_n|^d\right| > \frac{1}{2}|E_n|^d\right)\leq
\frac{\Var(|A_n|)}{\frac{1}{4}|E_n|^{2d}} \leq
\frac{4|E_n|^d(1-|E_n|^{d-1})}{|E_n|^{2d}}\tendvers 0.\]
\end{proof}

\begin{prop} \label{densBernoulli} Let $\bs{A}$ be a sequence of Bernoulli random subsets with density $d$. If $d\neq 0$, then $\bs A$ is densable and:
\[\dens\bs{A} = 
    \begin{cases}   d & \textup{ if } \; 0<d\leq 1 \\
                    -\infty & \textup{ if }\;  d<0.
    \end{cases}\]
\end{prop}

\begin{proof}
\begin{enumerate} [(i)]
    \item If $d<0$, by Markov's inequality \[\Pr(|A_n|\geq 1) \leq |E_n|^d\to 0,\] so $\Pr(A_n = \vide)\to 1$ and $\Pr(\log_{|E_n|}|A_n| = -\infty)\to 1$.
    \item If $0< d\leq 1$, by Lemma \ref{concentration Bernoulli} a.a.s.    \[\frac{1}{2}|E_n|^d\leq|A_n|\leq\frac{3}{2}|E_n|^d.\]
    For every $\varepsilon>0$, the inequality $|E_n|^{d-\varepsilon} < \frac{1}{2}|E_n|^d < \frac{3}{2}|E_n|^d < |E_n|^{d+\varepsilon}$ holds for $n$ large enough. Thus a.a.s.
    \[|E_n|^{d-\varepsilon}\leq|A_n|\leq|E_n|^{d+\varepsilon}.\]
    Hence $\bs A =(A_n)$ is densable with density $d$ by Proposition \ref{Concentration}.
\end{enumerate}
\end{proof}

\begin{thm}[The intersection formula for Bernoulli density model]\label{intersection Bernoulli} Let $\bs A,\bs B$ be independent sequences of Bernoulli random subsets of $\bs E = (E_n)$ with densities $\alpha,\beta$. Then $\bs{A}\cap\bs{B}$ is a sequence of Bernoulli random subsets of $\bs E$ with density $\alpha+\beta-1$, and
\[\dens (\bs A\cap \bs B) = \begin{cases}
\alpha + \beta - 1 & \textup{ if } \alpha + \beta > 1\\
-\infty & \textup{ if } \alpha + \beta < 1.
\end{cases}
\]
\end{thm}
\begin{proof}
For every elements $x\in E_n$, $\Pr(x\in A_n\cap B_n) = \Pr(x\in A_n)\Pr(x\in B_n) = |E_n|^{(\alpha+\beta-1)-1}$. In addition, for every pair of distinct elements $x,y$ in $E_n$
\[\begin{split}\Pr(x,y\in A_n\cap B_n) & = \Pr(x,y\in A_n)\Pr(x,y\in B_n)\\
& = \Pr(x\in A_n)\Pr(y\in A_n)\Pr(x\in B_n)\Pr(y\in B_n)\\
& = \Pr(x\in A_n\cap B_n)\Pr(y\in A_n\cap B_n).
\end{split}\]
So $\bs{A}\cap\bs{B}$ is a sequence of Bernoulli random subsets with density $\alpha+\beta-1$. Proposition \ref{densBernoulli} gives its density.
\end{proof}

As the theorem shows, the class of Bernoulli random subsets is \textit{closed under intersections}. Thereby the intersection formula works for multiple independent sequences of random subsets. The formula is more concise in terms of \textit{codensities}.

\begin{defi}[c.f. \cite{Gro93} p.269]
Let $\bs A$ be a densable sequence of random subsets such that $\dens \bs A\in [0,1]$. Then the codensity of $\bs A$ is defined by:
\[\codens{\bs A} = 1- \dens{\bs A}.\]
\end{defi}

Theorem \ref{intersection Bernoulli} can be rephrase as (compare \cite{Gro93} p.270):

\begin{thmp}{\ref{intersection Bernoulli}'}[The intersection formula by codensities] Let $\bs A,\bs B$ be independent sequences of Bernoulli random subsets of $\bs E$ with positive densities. If $\codens \bs A + \codens\bs B<1$, then
\[\codens (\bs A\cap \bs B) = \codens \bs A + \codens\bs B.\]

 If $\codens \bs A + \codens\bs B > 1$, then $\dens(\bs A \cap \bs B) = -\infty$.
\end{thmp}

\begin{cor}[Generalized intersection formula by codensities] Let $\bs A_1, \dots, \bs A_k$ be independent sequences of Bernoulli random subsets with positive densities. If $\displaystyle{\sum_{i=1}^k\codens \bs A_i<1}$, then
\[\codens\left(\bigcap_{i=1}^k \bs A_i\right) = \sum_{i=1}^k\codens \bs A_i.\]

If $\displaystyle{\sum_{i=1}^k\codens \bs A_i > 1}$, then $\dens\left(\bigcap_{i=1}^k \bs A_i\right) = -\infty$.
\end{cor}\quad\\

As an exception, a Bernoulli sequence of random subsets with density $d = 0$ is \textit{not densable}.

\begin{prop}\label{dens0Bernoulli} Let $\bs A$ be a Bernoulli sequence with density $d = 0$. Then $\bs A$ is \textbf{not} densable. In fact,
\[\Pr(\dens A_n = -\infty) \tendvers 1/e\; .\]
\end{prop}
\begin{proof}
$\Pr(|A_n|=0)=(1-|E_n|\moinsun)^{|E_n|}\tendvers 1/e$, which gives
\[\Pr(\dens A_n = -\infty) \tendvers 1/e.\]
This justifies that the sequence of random variables $(\dens_{E_n} A_n)$ does not converge to any constant distribution.
\end{proof}
\subsection{The uniform density model}

The uniform density model is the first example of densable sequences of random subsets. It is introduced by M.Gromov \cite{Gro93} to construct random groups with fixed generators, and later developed by Y. Ollivier \cite{Oll05}. It is also used by A. \.Zuk \cite{Zuk03} to study random triangular groups.

Let $\bs E = (E_n)$ be a sequence of sets. To simplify, we assume that $|E_n|=n$ in this subsection. For an arbitrary sequence $\bs E$ with $|E_n|\to\infty$ we can proceed similar proofs by replacing $n$ by $|E_n|$. Note that $|E_n|^d = n^d\sim\lfloor n^d\rfloor$ while $n\to\infty$ for $d\in [0,1]$. 

Recall that a sequence of uniform random subsets (example \ref{ex densable sequences} (iv)) of $(E_n)$ with density $d$ is a sequence of random subsets $(A_n)$ with the following laws:
    \[\Pr(A_n=a)=\begin{cases}
    \binom{n}{\flr{n^d}}\moinsun & \textup{ if } |a| = \lfloor n^d\rfloor \\
    0 & \textup{ if } |a| \neq \lfloor n^d\rfloor.
    \end{cases}\] 

We give here a concentration lemma for uniform density model, similar to Lemma \ref{concentration Bernoulli}. For the proof we need Lemma \ref{closed} and Lemma \ref{E and Var} in the next section.

\begin{lem}[Concentration lemma for uniform density model]\label{concentration uniform} Let $\bs A,\bs B$ be independent sequences of uniform random subsets of $\bs E$ with densities $\alpha,\beta\in [0,1]$. Then:
\begin{enumerate}[$(i)$]
    \item $n^{\alpha+\beta-1}-2\leq \mathbb{E}(|A_n\cap B_n|) \leq n^{\alpha+\beta-1}$.
    \item If $\alpha<1$ and $\beta <1$, then $\Var(|A_n\cap B_n|) \sim n^{\alpha+\beta-1}$.\\[6pt]
   Moreover, if $n\geq 3$, then $\Var(|A_n\cap B_n|) \leq 3n^{\alpha+\beta-1}$.
    \item Let $0<c<1$. If $\alpha+\beta - 1 > 0$ and $n\geq \left(\frac{4}{c}\right)^{\frac{1}{\alpha+\beta-1}}$, then
    \[\Pr\left(\left||A_n\cap B_n|-n^{\alpha+\beta-1}\right|> cn^{\alpha+\beta-1}\right)\leq \frac{12}{c^2n^{\alpha+\beta-1}}\tendvers 0.\] 
    In particular, a.a.s.
    \[\left||A_n\cap B_n|-n^{\alpha+\beta-1}\right| \leq cn^{\alpha+\beta-1}.\]
\end{enumerate}
\end{lem}

\begin{proof}\quad
\begin{enumerate}[$(i)$]
    \item By Lemma \ref{closed}, $A_n\cap B_n$ is a permutation invariant random set of $E_n$. Apply Lemma \ref{E and Var}:
    \[\begin{split}\mathbb{E}(|A_n\cap B_n|)& = n \Pr(x\in A_n\cap B_n) = n \Pr(x\in A_n)\Pr(x\in B_n)\\
    & = n\frac{\mathbb{E}(|A_n|)}{n}\frac{\mathbb{E}(|B_n|)}{n}
    = \flr{n^\alpha}\flr{n^\beta}n\moinsun \sim n^{\alpha + \beta -1}. \end{split}\]
    For the inequality, as $\alpha, \beta\leq 1$:
    \[n^{\alpha+\beta-1}-2\leq n^{\alpha+\beta-1} - n^{\alpha-1} - n^{\beta-1} + n\moinsun\leq  \flr{n^\alpha}\flr{n^\beta}n\moinsun \leq n^{\alpha+\beta-1}.\]
    \item Let $x,y$ be distinct elemensts in $E$. The number of subsets of $E$ containing $x,y$ of cardinality $\flr{n^\alpha}$ is $\binom{n-2}{\flr{n^\alpha}-2}$, so
    \[\Pr(x,y\in A_n) = \frac{\binom{n-2}{\flr{n^\alpha}-2}}{\binom{n}{\flr{n^\alpha}}} = \frac{\flr{n^\alpha}(\flr{n^\alpha}-1)}{n(n-1)}.\]
    Similarly,
    \[\Pr(x,y\in B_n) = \frac{\flr{n^\beta}(\flr{n^\beta}-1)}{n(n-1)}.\]
    Denote $k=\flr{n^\alpha}$ and $l=\flr{n^\beta}$ to simplify the notation. Note that $k = o(n)$ and $l=o(n)$ as $\alpha<1$ and $\beta<1$. Recall that $\mathbb{E}(|A_n\cap B_n|) = kln\moinsun$. Apply Lemma \ref{E and Var}, the variance of $|A_n\cap B_n|$ is
    \[\begin{split}
        \Var(|A_n\cap B_n|) & = kln\moinsun +n(n-1)\Pr(x,y\in A_n)\Pr(x,y\in B_n)-(kln\moinsun)^2\\
        & = kln\moinsun +\frac{k(k-1)l(l-1)}{n(n-1)}-(kln\moinsun)^2\\
        & = \frac{kl}{n^2(n-1)}(n^2-n+nkl-nl-nk+n-nkl+kl)\\
        & \sim \frac{kl}{n^2(n-1)}\cdot n^2 \sim n^{\alpha+\beta-1}.
        \end{split}\]
    Moreover, if $n\geq 3$ then:
    \[\Var(|A_n\cap B_n|) = \frac{kl}{n^2(n-1)}(n^2-nl-nk+kl)\\
    \leq \frac{2kl}{n-1}\leq \frac{2n^{\alpha+\beta}}{n-1}\leq 3n^{\alpha+\beta-1}.\]
    
    \item By $(i)$ if $n\geq \left(\frac{4}{c}\right)^{\frac{1}{\alpha+\beta-1}}\geq 4$, then\[\left|\mathbb{E}(|A_n\cap B_n|)-n^{\alpha+\beta-1}\right|\leq \frac{c}{2}n^{\alpha+\beta-1}.\]
    If $\alpha = 1$ or $\beta = 1$ then the result is true as the $A_n = E_n$ or $B_n=E_n$. Otherwise by $(ii)$ and Chebyshev's inequality, if $n\geq \left(\frac{4}{c}\right)^{\frac{1}{\alpha+\beta-1}}$ then
    \[\begin{split} &\quad\Pr\left(\left||A_n\cap B_n|-n^{\alpha+\beta-1}\right|> cn^{\alpha+\beta-1}\right)\\
    & \leq \Pr\left(\left||A_n\cap B_n|-\mathbb{E}(|A_n\cap B_n|)\right|> \frac{c}{2}n^{\alpha+\beta-1}\right)\\
    & \leq \frac{4\Var(|A_n\cap B_n|)}{c^2n^{2\alpha+2\beta-2}} \leq  \frac{12}{c^2n^{\alpha+\beta-1}}.
    \end{split}\]
\end{enumerate}
\end{proof}

\begin{prop} [The intersection formula for uniform density model]\label{intersection uniform}
Let $\bs A,\bs B$ be independent sequences of uniform random subsets of $\bs E$ with densities $\alpha,\beta$. If $\alpha + \beta \neq 1$, then the sequence $\bs A\cap \bs B$ is densable and
        \[\dens(\bs A\cap \bs B) = \begin{cases}
            \alpha + \beta -1 &  \textup{ if } \alpha + \beta > 1\\
            -\infty & \textup{ if } \alpha + \beta < 1.
        \end{cases}\]
\end{prop}
\begin{proof}\quad
    \begin{enumerate}[(i)]
        \item If $\alpha + \beta <1$, then by Markov's inequality and Lemma \ref{concentration uniform} $(i)$:
        \[\Pr(|A_n\cap B_n|\geq 1)\leq \mathbb{E}(|A_n\cap B_n|)\tendvers 0,\]
        which implies a.a.s. $A_n\cap B_n = \vide$ and $\dens(A\cap B) = -\infty$.
        \item If $\alpha + \beta >1$, by Lemma \ref{concentration uniform} $(iii)$ (with $c = 1/2$) a.a.s.
        \[\left||A_n\cap B_n|-n^{\alpha+\beta-1}\right|\leq \frac{1}{2}n^{\alpha+\beta-1},\]
        so for all $\varepsilon>0$ a.a.s.
        \[n^{\alpha+\beta-1-\varepsilon} \leq |A_n\cap B_n|\leq n^{\alpha+\beta-1+\varepsilon}.\]
        Hence by Proposition \ref{Concentration} $\bs A \cap \bs B$ is densable with density $\alpha + \beta -1$.
    \end{enumerate}
\end{proof}

The cardinality of $A_n\cap B_n$ is close to $n^{\alpha+\beta-1}$ with high probability, but not always. If $\alpha\neq 1$ and $\beta\neq 1$, then for $n$ large enough $\flr{n^\alpha} + \flr{n^\beta} < n$, so $\Pr(A_n\cap B_n = \vide)\neq 0$.

Which means that $\bs A\cap \bs B$ is not a sequence of uniform random subsets, so the class of sequences of uniform random subsets is \textit{not} closed under intersection.

\section{The general model: densable and permutation invariant}
\subsection{Densable sequences of permutation invariant random subsets}

Let $E$ be a finite set with cardinality $|E| = n$. Denote $\mathcal{S}(E)$ as the group of permutations of $E$. The action of $\mathcal{S}(E)$ on $E$ can be extended on $\mathcal{P}(E)$, defined by $\sigma(\{x_1,\dots,x_k\}) := \{\sigma(x_1),\dots, \sigma(x_k)\}$.

Note that this action has $(n+1)$ orbits of the form $\{a\in \mathcal{S}(E)\mid |a| = k\}$ for $k \in \{ 0,\dots,n\}$. Moreover, the action commutes with set theoretic operations: $\sigma(E\backslash a) = E\backslash \sigma(a)$, $\sigma(a\cap b) = \sigma(a)\cap\sigma(b)$ and $\sigma(a\cup b) = \sigma(a)\cup\sigma(b)$.

\begin{defi}[Permutation invariant random subsets]\quad
Let $A$ be a random subset of $E$. It is permutation invariant if its law is invariant by the permutations of $E$. i.e.
    \[\forall a\in \mathcal{P}(E)\;\forall \sigma\in\mathcal{S}(E)\quad \Pr(A = a) = \Pr(A = \sigma(a)).\]
\end{defi}

Equivalently, subsets of $E$ of the same cardinality are equiprobable. There exists real numbers $p_0,\dots,p_n\in [0,1]$ satisfying 
\[\sum_{k=0}^n\binom{n}{k} p_k= 1\] such that
    \[\forall a\in\mathcal{P}(E) \quad |a|= k\Rightarrow\Pr(A = a) = p_k.\]

By definition, uniform random subsets and Bernoulli random subsets are permutation invariant. The advantage of such class of random subsets is that it is closed under set theoretic operations, especially under intersections.

\begin{lem} [Closed under set operations]\label{closed} Let $E$ be a finite set. The class of permutation invariant random subsets of $E$ is closed under set theoretic operations (union, complement and intersection).
\end{lem}
\begin{proof}\quad
\begin{enumerate}[$(i)$]
    \item (Complement) Let $A$ be a permutation invariant random subset. Let $a\in \mathcal{P}(E)$ and $\sigma\in\mathcal{S}(E)$. Then
    \[\begin{split}
        \Pr(E\backslash A = a) & = \Pr(A = E\backslash a) = \Pr(A = \sigma(E\backslash a))\\
        & = \Pr(A = E\backslash \sigma(a)) = \Pr(E\backslash A = \sigma(a)).
    \end{split}\]
    \item (Intersection) Let $A,B$ be independent permutation invariant random subsets. Then for $\sigma\in\mathcal{S}(E)$
    \[
    \begin{split}\Pr(A\cap B = c) & = \sum_{a,b\in\mathcal{P}(E);a\cap b = c}\Pr(A = a)\Pr(B = b) \\
    & = \sum_{a,b\in\mathcal{P}(E);\sigma(a)\cap \sigma(b) = \sigma(c)}\Pr(A = \sigma(a))\Pr(B = \sigma(b))\\
    & = \sum_{a',b'\in\mathcal{P}(E);a'\cap b' = \sigma(c)}\Pr(A = a')\Pr(B = b') \quad (\textup{by substitution})\\
    & = \Pr(A\cap B = \sigma(c)).
    \end{split}\]
    \item (Union) Let $A,B$ be independent permutation invariant random subsets. Then $A\cup B = E\backslash((E\backslash A)\cap (E\backslash B))$. So $A\cup B$ is permutation invariant.
\end{enumerate}
\end{proof}

We shall express the expectation and the variance of the random variable $|A|$ by $\Pr(x\in A)$ and $\Pr(x\in A,y\in A)$ where $x,y$ are distinct elements in $E$.
\begin{lem} \label{E and Var} Let $A$ be a permutation invariant random subset of $E$. Let $x,y$ be distinct elements in $E$. Then
\begin{enumerate}[$(i)$]
    \item $\mathbb{E}(|A|) =n \Pr(x\in A)$,
    \item $\Var(|A|) = \mathbb{E}(|A|) +n(n-1)\Pr(x\in A,y\in A)-\mathbb{E}(|A|)^2$.
\end{enumerate}
\end{lem}
\begin{proof}\quad
\begin{enumerate}[$(i)$]
    \item By definition the probability $\Pr(z\in A)$ does not depend on the choice of element $z\in E$. So \[\mathbb{E}(|A|) = \mathbb{E}\left(\sum_{z\in E}\mathbbm{1}_{z\in A}\right) = \sum_{z\in E}\Pr(z\in A) = n \Pr(x\in A).\]
    \item By the same argument, the probability $\Pr(z\in A, w\in A)$ does not depend on the choice of pair of distinct elements $(z,w)$ in $E$. So
    \[\begin{split}\mathbb{E}(|A|^2) & = \mathbb{E}\left[\left(\sum_{z\in E}\mathbbm{1}_{z\in A}\right)^2\right]\\
    & = \sum_{z\in E}\Pr(z\in A) + \sum_{(z,w)\in E^2;z\neq w} \Pr(z\in A, w\in A)\\
    & = \mathbb{E}(|A|) +n(n-1)\Pr(x\in A,y\in A). \end{split}\]
\end{enumerate}
\end{proof}

A permutation invariant random subset can be decomposed into uniform random subsets.

\begin{prop} [Decomposition into uniform random subsets]\label{decomposition} Let $A$ be a permutation invariant random subset of $E$. 
\begin{enumerate} [(i)]
    \item If $\Pr(|A| = k)\neq 0$, then the random subset $A$ under the condition $\{|A| = k\}$ is a uniform random subset on all subsets of $E$ of cardinality $k$.
    \item Let $Q$ be an event described by $A$ (for example, $Q = \{x\in A\}$). Denote $\mathbb{N}_A = \{k\in \mathbb{N}\,|\,\Pr(|A|=k)\neq 0\}$, then
\[\Pr(Q) = \sum_{k\in \mathbb{N}_A}\Pr(Q\mid|A| = k)\Pr(|A|=k).\]
\end{enumerate}
\end{prop}
\begin{proof}
    Suppose that $\Pr(|A| = k)\neq 0$. Let $a\subset E$ of cardinal $k$. As $A$ is permutation invariant,
    \[\Pr(|A| = k) = \binom{n}{k}\Pr(A = a).\]
    Hence 
    \[\Pr\left(A = a\,\middle|\,|A| = k\right) = \frac{\Pr(A = a)}{\Pr(|A| = k)} = \binom{n}{k}\moinsun.\]
    If $|a|\neq k$ then $\Pr\left(A = a\,\middle|\,|A| = k\right) = 0$.
    
    The second assertion is the formula of total probability.
\end{proof}

\begin{defi} Let $\bs A = (A_n)$ be a sequence of random subsets of $\bs E = (E_n)$. It is a sequence of permutation invariant random subset if $A_n$ is a permutation invariant random subset of $E_n$ for all $n$. 
\end{defi}
\begin{nota} Let $\bs E = (E_n)$ be a sequence of finite sets. Denote $\mathcal{D}(\bs E)$ the class of densable sequences of permutation invariant random subsets of $\bs E$.
\end{nota}

\begin{ex}\quad
\begin{enumerate}
    \item Sequences of Bernoulli random subsets of $\bs E$ with density $d\neq 0 $ are in the class $\mathcal{D}(\bs E)$.
    \item Sequences of uniform random subsets of $\bs E$ are in the class $\mathcal{D}(\bs E)$. 
    \item Let $\bs A,\bs B$ be independent sequences of uniform random subsets. By Lemma \ref{closed}, the sequence $\bs A\cap \bs B$ is permutation invariant. By Proposition \ref{intersection uniform}, if $\dens\bs A + \dens\bs B\neq 1$, then $\bs A\cap \bs B$ is densable. In this case the sequence $\bs A\cap \bs B$ is in the class $\mathcal{D}(\bs E)$.
\end{enumerate}
\end{ex}

Except for some special cases, the class $\mathcal{D}(\bs E)$ is closed under set theoretic operations:

\begin{prop} Let $\bs A, \bs B\in \mathcal{D}(\bs E)$ with densities $\alpha,\beta$. Then the union $\bs A\cup\bs B$ is in $\mathcal{D}(\bs E)$ and $\dens(\bs A\cup\bs B) = \max(\alpha,\beta)$.
\end{prop}
\begin{proof}
    By Lemma \ref{closed} the sequence of random subset $\bs A\cup\bs B$ is permutation invariant. The cases $\alpha=0$ or $\beta=0$ can be easily shown. Without loss of generality, assume that $\alpha \geq \beta \geq 0$.
    
    Let $\varepsilon>0$. By densabilities of $\bs A$ and $\bs B$, a.a.s.
    \[n^{\alpha-\varepsilon/2}\leq |A_n|\leq n^{\alpha+\varepsilon/2},\]
    \[n^{\beta-\varepsilon/2}\leq |B_n|\leq n^{\beta+\varepsilon/2}.\]
    Thus a.a.s.
    \[n^{\alpha-\varepsilon}\leq |A_n|\leq |A_n\cup B_n|\leq n^{\alpha+\varepsilon/2}+n^{\beta+\varepsilon/2}\leq 2n^{\alpha+\varepsilon/2}\leq n^{\alpha+\varepsilon}.\]
\end{proof}

\begin{prop}\label{complement}
    Let $\bs A\in \mathcal{D}(\bs E)$ with density $\alpha < 1$. Then the complement $\bs E\backslash \bs A$ is in $\mathcal{D}(\bs E)$ and $\dens(\bs E\backslash \bs A) = 1$.
\end{prop}
\begin{proof}
    Again by Lemma \ref{closed} the sequence of random subset $\bs E\backslash \bs A$ is permutation invariant.
    
    Let $0<\varepsilon<(1-\alpha)/2$. By densablility of $\bs A$, a.a.s.
    \[|A_n|\leq n^{\alpha+\varepsilon}.\]
    As $n^{\alpha+\varepsilon}+n^{1-\varepsilon}\leq n$ for $n$ large enough, a.a.s.
    \[|E_n\backslash A_n| \geq n-n^{\alpha+\varepsilon}\geq n^{1-\varepsilon}.\]
\end{proof}

\subsection{The intersection formula}

In this subsection we shall prove the intersection formula for the class of densable sequences of permutation invariant random subsets.

\begin{thm}[The intersection formula]\label{intersection} Let $\bs A, \bs B$ be independent sequences in $\mathcal{D}(\bs E)$ with densities $\alpha,\beta$. If $\alpha + \beta \neq 1$, then the sequence $\bs A\cap \bs B$ is in $\mathcal{D}(\bs E)$ and 
    \[\dens (\bs A\cap \bs B) = \begin{cases}
        \alpha + \beta - 1 & \textup{ if } \alpha + \beta > 1\\
        -\infty & \textup{ if } \alpha + \beta < 1.
        \end{cases}\]
\end{thm}

\begin{lem}\label{aux intersection} Let $\alpha,\beta\in [0,1]$ such that $\alpha + \beta > 1$. Let $0<\varepsilon<\alpha+\beta-1$. Let $\bs A, \bs B$ independent sequences of uniform random subsets of $\bs E$ with densities $\alpha',\beta'$ with $\alpha'\in [\alpha-\varepsilon/3,  \alpha+\varepsilon/3]$ and $\beta'\in [\beta-\varepsilon/3, \beta+\varepsilon/3]$. If $n\geq \max\left\{2^{3/\varepsilon},8^{1/(\alpha+\beta-1-\varepsilon)}\right\}$, then:
\[\Pr\left(n^{\alpha+\beta-1-\varepsilon}\leq |A_n\cap B_n|\leq n^{\alpha+\beta-1+\varepsilon}\right)\geq 1-\frac{48}{n^{\alpha+\beta-1-\varepsilon}}\tendvers 1.\]
\end{lem}
\begin{proof}
    By hypothesis $\alpha'+\beta'-1\geq\alpha+\beta-2\varepsilon/3-1>0$. Apply Lemma \ref{concentration uniform} $(iii)$ with $c = \frac{1}{2}$, for $n\geq 8^{1/(\alpha+\beta-1-\varepsilon)}\geq 8^{1/(\alpha'+\beta'-1)}$:
    \[\Pr\left(\left||A_n\cap B_n|-n^{\alpha'+\beta'-1}\right|\geq \frac{1}{2}n^{\alpha'+\beta'-1}\right)\leq  \frac{48}{n^{\alpha'+\beta'-1}}.\]
    This can be rewrite as 
    \[\Pr\left(\frac{1}{2}n^{\alpha' + \beta' -1} < |A_n\cap B_n| < \frac{3}{2}n^{\alpha' + \beta' -1} \right) > 1-\frac{48}{n^{\alpha'+\beta'-1}}.\]
    Again by hypothesis $\alpha + \beta -1 - 2\varepsilon/3\leq \alpha'+\beta'-1\leq \alpha + \beta -1 + 2\varepsilon/3$. If $n\geq 2^{3/\varepsilon}$, then
    \[n^{\alpha+\beta-1-\varepsilon}\leq \frac{1}{2}n^{\alpha + \beta -1 - 2\varepsilon/3} \leq \frac{3}{2}n^{\alpha + \beta -1 + 2\varepsilon/3} \leq n^{\alpha+\beta-1+\varepsilon},\]
    so:
    \[\begin{split}&\Pr\left(n^{\alpha+\beta-1-\varepsilon}\leq |A_n\cap B_n|\leq n^{\alpha+\beta-1+\varepsilon} \right)\\
    \geq & \Pr\left(\frac{1}{2}n^{\alpha + \beta -1 - 2\varepsilon/3}\leq |A_n\cap B_n|\leq \frac{3}{2}n^{\alpha + \beta -1 + 2\varepsilon/3} \right)\\
    \geq & \Pr\left(\frac{1}{2}n^{\alpha' + \beta' -1} < |A_n\cap B_n| < \frac{3}{2}n^{\alpha' + \beta' -1} \right).
    \end{split}\]
    Combine two estimations on $n$. If $n\geq \max\left\{2^{3/\varepsilon},8^{1/(\alpha+\beta-1-\varepsilon)}\right\}$, then:
    \[\Pr\left(n^{\alpha+\beta-1-\varepsilon}\leq |A_n\cap B_n|\leq n^{\alpha+\beta-1+\varepsilon}\right)\geq 1-\frac{48}{n^{\alpha'+\beta'-1}}\geq  1-\frac{48}{n^{\alpha+\beta-1-\varepsilon}}.\]
    As $\alpha+\beta-1-\varepsilon>0$, when $n$ goes to infinity
    \[\frac{48}{n^{\alpha+\beta-1-\varepsilon}}\tendvers 0.\]
\end{proof}

\begin{proof}[Proof of Theorem \ref{intersection}]
    By Lemma \ref{closed} the intersection $\bs A\cap \bs B$ is a sequence of permutation invariant random subsets. In either cases, denote $(Q_n)$ the sequence of events defined by
    \[Q_n = \{n^{\alpha-\varepsilon/3}\leq|A_n|\leq n^{\alpha+\varepsilon/3}\textup{ and }n^{\beta-\varepsilon/3}\leq|B_n|\leq n^{\beta+\varepsilon/3}\}\] 
    for some small $\varepsilon>0$. By the densabilities of $\bs A$ and $\bs B$, a.a.s. $Q_n$ is true. Note that $Q_n$ is a union of events of type $\{|A_n|= k,|B_n| = l\}$. Denote 
    \[\begin{split}
        \mathbb{N}^2_{\bs A,\bs B, n, \varepsilon} := & \left\{ (k,l)\in \mathbb{N}^2 \,\middle|\, n^{\alpha-\varepsilon/3}\leq k \leq n^{\alpha+\varepsilon/3},n^{\beta-\varepsilon/3}\leq l \leq n^{\beta+\varepsilon/3} \right.\\
    & \textup{ and }\Pr(|A_n| = k,|B_n|=l)\neq 0 \Big\}.
    \end{split}\]
    For $(k,l)\in \mathbb{N}^2_{\bs A,\bs B, n, \varepsilon}$, we may do a change of variables $k = n^{\alpha'}$, $l = n^{\beta'}$ so that
        \[\alpha-\varepsilon/3\leq \alpha'\leq \alpha+\varepsilon/3\textup{ and }\beta-\varepsilon/3\leq \beta'\leq \beta+\varepsilon/3.\]
    
    \begin{enumerate}[(i)]
        \item Suppose that $\alpha +\beta<1$. Let $0<\varepsilon<1-\alpha-\beta$. We shall prove that a.a.s. $A_n\cap B_n=\vide$.
        
        By the formula of total probability and Markov's inequality,
        \[ \begin{split}
        \Pr(A_n\cap B_n\neq\vide) & \leq \Pr\left(|A_n \cap B_n|\geq 1\,\middle|\,Q_n\right)\Pr(Q_n)+\Pr(\overline{Q_n})\\
        & \leq \sum_{(k,l)\in\mathbb{N}^2_{\bs A,\bs B, n, \varepsilon}} \Big[ \Pr\left(|A_n \cap B_n|\geq 1\,\middle|\,|A_n|=k,|B_n| = l\right)\\
        &\qquad\qquad\qquad\quad\Pr(|A_n|=k,|B_n| = l|)\Big]+\Pr(\overline{Q_n}).\\        
        & \leq \sum_{(k,l)\in\mathbb{N}^2_{\bs A,\bs B, n, \varepsilon}} \Big[ \esp{|A_n \cap B_n|\,\middle|\,|A_n|=k,|B_n| = l}\\
        &\qquad\qquad\qquad\quad\Pr(|A_n|=k,|B_n| = l|)\Big]+\Pr(\overline{Q_n}).
        \end{split}\]
        
        For any $(k,l)\in \mathbb{N}^2_{\bs A,\bs B, n, \varepsilon}$, by Lemma \ref{concentration uniform} (i)
        \[\begin{split}
        \esp{|A_n \cap B_n|\,\middle|\,|A_n|=k,|B_n| = l}
        & = \esp{|A_n \cap B_n|\,\middle|\,|A_n|=n^{\alpha'},|B_n| = n^{\beta'}}\\
        & \leq n^{\alpha'+\beta'-1} \leq n^{\alpha+\beta+2/3\varepsilon-1}\leq n^{-1/3\varepsilon}. 
        \end{split}\]
        Hence
        \[\Pr(A_n\cap B_n\neq\vide) \leq n^{-1/3\varepsilon}\Pr(Q_n)+\Pr(\overline{Q_n})\tendvers 0.\]
        
        \item Suppose that $\alpha+\beta>1$. Let $0<\varepsilon<\alpha+\beta-1$. We shall prove that a.a.s.
        \[n^{\alpha+\beta-1-\varepsilon}\leq|A_n\cap B_n|\leq n^{\alpha+\beta-1+\varepsilon}.\]

        By the formula of total probability,
        \[ \begin{split}&\quad \Pr(n^{\alpha+\beta-1-\varepsilon} \leq |A_n\cap B_n| \leq n^{\alpha+\beta-1+\varepsilon})\\
        & \geq \Pr\left(n^{\alpha+\beta-1-\varepsilon}\leq |A_n\cap B_n| \leq n^{\alpha+\beta-1+\varepsilon}\,\middle|\,Q_n\right)\Pr(Q_n)\\
        & = \sum_{(k,l)\in\mathbb{N}^2_{\bs A,\bs B, n, \varepsilon}} \Big[ \Pr\left(n^{\alpha+\beta-1-\varepsilon}\leq |A_n\cap B_n| \leq n^{\alpha+\beta-1+\varepsilon}\,\middle|\,|A_n|=k,|B_n| = l\right)\\
        &\qquad\qquad\qquad\quad\Pr(|A_n|=k,|B_n| = l|)\Big].
        \end{split}\]
        
        By Lemma \ref{aux intersection} and Proposition \ref{decomposition}. If $n\geq \max\left\{2^{3/\varepsilon},8^{1/(\alpha+\beta-1-\varepsilon)}\right\}$, then for any $(k,l)\in \mathbb{N}^2_{\bs A,\bs B, n, \varepsilon}$:
        \[\begin{split} &\quad \Pr\left(n^{\alpha+\beta-1-\varepsilon}\leq |A_n\cap B_n|\leq n^{\alpha+\beta-1+\varepsilon} \,\middle|\, |A_n|= k,|B_n| = l \right)\\
        & = \Pr\left(n^{\alpha+\beta-1-\varepsilon}\leq |A_n\cap B_n|\leq n^{\alpha+\beta-1+\varepsilon} \,\middle|\, |A_n|=n^{\alpha'},|B_n| = n^{\beta'}\right)\\
        & \geq 1-\frac{48}{n^{\alpha+\beta-1+\varepsilon}} \tendvers 1.
        \end{split}
        \]
        Hence for $n\geq \max\left\{2^{3/\varepsilon},8^{1/(\alpha+\beta-1-\varepsilon)}\right\}$:
        \[\begin{split} & \quad \Pr(n^{\alpha+\beta-1-\varepsilon}\leq|A_n\cap B_n|\leq n^{\alpha+\beta-1+\varepsilon})\\
        &\geq \sum_{(k,l)\in\mathbb{N}^2_{\bs A,\bs B, n, \varepsilon}} \left( 1-\frac{48}{n^{\alpha+\beta-1+\varepsilon}} \right) \Pr(|A_n|=k, |B_n|=l)\\
        &\geq \left(1-\frac{48}{n^{\alpha+\beta-1+\varepsilon}}\right)\Pr(Q_n)\tendvers 1.
        \end{split}\]
            
    \end{enumerate}
\end{proof}

Remark that when $\alpha+\beta = 1$ the density is not determined , as Proposition \ref{dens0Bernoulli} showed for Bernoulli random subsets. As the class is closed under intersection, we can conclude on multiple intersections.

\begin{cor} Let $\bs A_1, \dots,\bs A_k$ be independent sequences in $\mathcal{D}(\bs E)$ of positive densities. If $\displaystyle{\sum_{i=1}^k\codens \bs A_i<1}$, then
\[\codens\left(\bigcap_{i=1}^k \bs A_i\right) = \sum_{i=1}^k\codens \bs A_i.\]

If $\displaystyle{\sum_{i=1}^k\codens \bs A_i > 1}$, then $\dens\left(\bigcap_{i=1}^k \bs A_i\right) = -\infty$.
\end{cor}\quad\\

\subsection{Another model: random functions}

We give here another natural model of random subsets : image of a random function, which can be found in \cite{Gro93} p.271 by Gromov. This is also a variance of random groups considered by Ollivier in \cite{Oll05} Lemma 59. In this subsection we prove that such a model is densable and permutation invariant.

\begin{defi} Let $E,F$ be finite subsets of cardinalities $n,m$. Denote $E^F$ the set of functions from $F$ to $E$. A random function $\Phi$ from $F$ to $E$ is a $E^F$-valued random variable.
\end{defi}

Let $\Phi$ be a random function from $F$ to $E$. Its law is determined by \[\Pr(\Phi = \varphi)\] through all $\varphi\in E^F$.

The random function $\Phi$ can be regarded as a vector of $E$-valued random variables (or random elements of $E$) $(\Phi(y))_{y\in F}$ indexed by $F$. Note that these random elements are not necessarily independent. The image $\im(\Phi) = \Phi(F):=\{\Phi(y)|y\in F\}$ is then a random subset of $E$.

\begin{ex} (Uniform random function) Let $\Phi$ be the uniform distribution on all functions from $F$ to $E$. Its law is
\[\Pr(\Phi = \varphi) = \frac{1}{|E^F|} = \frac{1}{n^m}\]
through all $\varphi\in E^F.$
\end{ex}

\begin{prop} Let $\Phi$ be a uniform random function from $F$ to $E$. Then the random elements $(\Phi(y))_{y\in F}$ are independent (identical) uniform distributions on $E$.
\end{prop}
\begin{proof}
    Let $x\in E$, $y\in F$. The number of functions from $F$ to $E$ such that $\phi(y) = x$ is $n^{m-1}$. So the law of $\Phi(y)$ is 
    \[\Pr(\Phi(y) = x) = \frac{n^{m-1}}{n^m} = \frac{1}{n}.\]
    Which is an uniform distribution on $E$.
    
    Denote $F = \{y_1,\dots,y_m\}$. Let $(x_1,\dots,x_m)$ a vector of $m$ elements in $E$. Let $\varphi\in E^F$ such that $\varphi(y_i) = x_i$ for all $1\leq i \leq m$. Then 
    \[\Pr\left(\bigwedge\limits_{i=1}^m\Phi(y_i) = x_i\right) = \Pr(\Phi = \varphi) =  \frac{1}{n^m} = \prod_{i=1}^m \Pr(\Phi(y_i) = x_i).\]
\end{proof}\quad

\begin{prop} The image of an uniform random function is a permutation invariant random subset.
\end{prop}
\begin{proof}
    Let $\Phi$ be an uniform random function from $F$ to $E$. Let $\sigma\in \mathcal{S}(E)$, then for all $\varphi\in E^F$:
    \[\Pr(\Phi = \varphi) = \Pr(\Phi = \sigma\circ\varphi) = \Pr(\sigma\moinsun\circ\Phi = \varphi).\]
    The random function $\sigma\moinsun\circ\Phi$ has the same law of $\Phi$. Now let $a\subset E$
    \[\Pr(\im(\Phi) = a) =\Pr(\im(\sigma\moinsun\circ\Phi) = a) =\Pr(\im(\Phi) = \sigma(a)).\]
\end{proof}

\section{The multi-dimensional intersection formula}

Let $\bs E = (E_n)$ be a sequence of finite sets with $|E_n| = n$ and $k\geq 2$ be an integer. The set of pairwise different $k$-tuples of $E_n$ is 
\[E_n^{(k)} := \{(x_1,\dots,x_k)\in E_n^k \;|\; x_i\neq x_j\;\forall i\neq j\}.\]
Denote $\bs E ^{(k)} = (E_n^{(k)})_{n\in \mathbb{N}}$.

Similarly, for a sequence of random subsets $\bs A = (A_n)$ of $\bs E$, we can define
\[A_n^{(k)} := \{(x_1,\dots,x_k)\in A_n^k \;|\; x_i\neq x_j\;\forall i\neq j\},\]
which is a random subset of $E_n^{(k)}$. Denote also $\bs A^{(k)} = (A_n^{(k)})$. We will establish an intersection formula between a sequence of random subsets of type $\bs A^{(k)}$ and a sequence of fixed subsets $\bs X = (X_n)$ of $\bs E^{(k)}$.

\begin{prop} 
    Let $\bs A$ be a densable sequence of random subsets of $\bs E$ with density $d>0$. Then $\bs A^{(k)}$ is a densable sequence of random subsets of $\bs E^{(k)}$ with density $d$. Namely,
    \[\dens_{\bs E^{(k)}} (\bs A^{(k)}) = \dens_{\bs E} (\bs A).\]
\end{prop}
\begin{proof}
    Note that $n^k-k^2(n-1)^k \leq |E_n^{(k)}|\leq n^k$, so $ |E_n^{(k)}|= n^{k+o(1)}$.

    Let $\varepsilon>0$. By densability a.a.s. $n^{d-\varepsilon/2}\leq |A_n| \leq n^{d+\varepsilon/2}$. By the same argument above a.a.s. $|A_n^{(k)}|= |A_n|^{k+o(1)}$ as random variables. Hence a.a.s. \[(n^k)^{d-\varepsilon/2+o(1)}\leq |A_n^{(k)}| \leq (n^k)^{d+\varepsilon/2+o(1)},\]
    so a.a.s.
    \[|E_n^{(k)}|^{d-\varepsilon}\leq |A_n^{(k)}| \leq |E_n^{(k)}|^{d+\varepsilon}.\]
\end{proof}

Although the densability is preserved, it is not the case for being permutation invariant. Given a permutation invariant random subset $A_n$ of $E_n$, the random subset $A_n^{(k)}$ is \textit{not} permutation invariant in $E_n^{(k)}$ for $k\geq 2$. See the following example.

\begin{ex} Let $(A_n)$ be a sequence of Bernoulli random subsets of $(E_n)$ with density $0<d<1$. Recall that subsets of the same cardinality have the same probability to be included in a permutation invariant random subset. Let $x_1,\dots,x_4$ be distinct elements in $E_n$. 
\[\Pr\left(\{(x_1,x_2),(x_3,x_4)\}\subset A_n^{(2)}\right) = \Pr\left(\{x_1,x_2,x_3,x_4\}\subset A_n\right) = n^{4(d-1)},\]
while
\[\Pr\left(\{(x_1,x_2),(x_2,x_3)\}\subset A_n^{(2)}\right) = \Pr\left(\{x_1,x_2,x_3\}\subset A_n\right) = n^{3(d-1)}.\qed\]
\end{ex}\quad

As a result the classical intersection formula (Theorem \ref{intersection}) can not be applied in this context. Actually, for $k\geq 2$ the intersection formula \textit{does not} work for some choices of $\bs X$. We give here a counter example. 

\begin{ex}\label{cex no intersection formula}
Let $\bs A$ be a sequence of random subsets in $\mathcal{D}(\bs E)$ with density $3/4$. Let $\bs X = (X_n)$ be a sequence of subsets defined by 
\[X_n = \{x_n\}\times (E_n\backslash\{x_n\})\subset E_n^{(2)}\]
with some $x_n\in E_n$. By its construction $\dens_{\bs E^{(2)}}(\bs X) = 1/2$, so we expected that $\dens(\bs A^{(2)} \cap \bs X) = 3/4+1/2-1=1/4$. However, we have \[\dens(\bs A^{(2)} \cap \bs X) = 0\]
because a.a.s. $A_n \cap \{x_n\} = \vide$.
\end{ex}\quad

For the intersection formula between $\bs A^{(k)}$ and $\bs X$, we need an additional condition on $\bs X$. More precisely, $\bs X$ can not have too much "self-intersection". We will discuss this condition in subsection 4.1.\\

Following the path for proving the intersection formula (Theorem \ref{intersection}), we shall study the case that $\bs A$ is a sequence of Bernoulli random subsets with density $d$ (subsection 4.2). We then adapt the proof for the uniform density model by estimating the probabilities $\Pr\left(\{x_1,\dots,x_r\}\subset A_n\right)$ (subsection 4.3).

For the general case (subsection 4.4), according to Proposition \ref{decomposition}, we can decompose a permutation invariant random subset into uniform random subsets. We then need to bound $|A_n^{(k)}\cap X_n|$ for sequences of uniform random subsets, uniformly in a small neighborhood of densities $d'\in [d-\varepsilon,d+\varepsilon]$.

\subsection{Statement of the theorem}

\begin{defi}[Self-intersection partition] Let $\bs X = (X_n)$ be a sequence of fixed subsets of $\bs E^{(k)}$ with density $\alpha$. For $0\leq i \leq k$, the $i$-th self-intersection of $X_n$ is
\[Y_{i,n} := \{(x,y)\in X_n^2\,|\, |x\cap y| = i\}\]
where $|x\cap y|$ is the number of common elements of $x = (x_1,\dots,x_k)$ and $y = (y_1,\dots,y_k)$. In particular $Y_{0,n}$ is the set of pairs of $X_n$ having no intersection. 

Note that $(Y_{i,n})_{0\leq i \leq k}$ is a partition of $X_n^2$, called the self-intersection partition of $X_n$. Namely,
\[X_n^2 = \bigsqcup_{i=0}^k Y_{i,n}.\]

Denote $\bs Y_i = (Y_{i,n})_{n\in \mathbb{N}}$ the $i$-th self intersection of $\bs X$, and  $(\bs Y_i)_{0\leq i \leq k}$ is called the self-intersection partition of $\bs X$. Namely,
\[\bs X^2 = \bigsqcup_{i=0}^k \bs Y_i.\]
\end{defi}\quad

Remark that the sequences $\bs X^2$ and $\bs Y_i$ are sequences of fixed subsets of $(\bs E^{(k)})^2 = \left((E_n^{(k)})^2\right)_{n\in \mathbb{N}}$. Note that $\dens_{(\bs E^{(k)})^2} (\bs X^2) = \dens_{\bs E^{(k)}} (\bs X) = \alpha$. To give a condition on $\bs Y_i$, we need the notion of \textit{upper density}, defined by an upper limit:
 
\begin{defi} Let $\bs Y = (Y_n)$ be a sequence of subsets of $\bs E = (E_n)$. The upper density of $\bs Y$ in $\bs E$ is 
\[\overline{\dens}_{\bs E}\bs Y := \varlimsup_{n\to\infty}\log_{|E_n|}(|Y_{n}|).\]
\end{defi}

We introduce here, for a sequence of densable fixed subsets $\bs X$ of $\bs E^{(k)}$ with density $\alpha$, the small self-intersection condition: 

\begin{defi}\label{multidim condition} Let $\bs X$ be a sequence of subsets of $\bs E^{(k)}$ with density $\alpha$ and let $(\bs Y_i)_{0\leq i\leq k}$ be its self-intersection partition. Let $d > 1-\alpha$. We say that $\bs X$ has $d$-small self-intersection if for every $1\leq i \leq k-1$
\begin{equation}
    \overline{\dens}_{(\bs E^{(k)})^2}\left(\bs Y_i\right) < \alpha  - (1-d) \times \frac{i}{2k}.
\end{equation}
\end{defi}

Remark that the right-hand side of inequality (\ref{multidim condition}) is between $0$ and $\alpha$ because $\alpha > 1-d > 0$. Note that $|Y_{k,n}| = |\{(x,y)\in X_n^2\,|\, x = y\}|= |X_n|$ so \[\dens_{(\bs E^{(k)})^2} \bs Y_k = \frac{\alpha}{2} < \alpha-(1-d)\frac{k}{2k},\] which verifies (\ref{multidim condition}) automatically. On the other hand, as densities of $\bs Y_i$ for $1\leq i\leq k$ are all smaller then $\alpha$ and $|Y_{0,n}| = |X_n^2|-\sum_{i=1}^k|Y_{i,n}|$, by Proposition \ref{complement} \[\dens \bs Y_0 = \dens \bs X^2 = \alpha.\]
\quad

The purpose of this section is to demonstrate the following theorem. 
\begin{thm}[Multi-dimensional intersection formula] \label{multidim intersection} Let $\bs A$ be a densable sequence of permutation invariant random subsets of $\bs E$ with density $0<d<1$. Let $\bs X = (X_n)$ be a sequence of (fixed) subsets of $\bs E^{(k)}$ with density $\alpha$.
    \begin{enumerate}[(i)]
        \item If $d + \alpha <1$, then $\bs A^{(k)}\cap \bs X$ is densable and \[\dens (\bs A^{(k)}\cap \bs X) = -\infty.\]
        \item If $d + \alpha >1$ and $\bs X$ has $d$-small self intersection (condition (\ref{multidim condition})), then $\bs A^{(k)}\cap \bs X$ is densable and 
        \[\dens (\bs A^{(k)}\cap \bs X) = \alpha + d - 1.\]
    \end{enumerate} 
\end{thm}

Note that by taking $k=1$, we have the intersection formula between a random subset and a fixed subset. In this case we do not need to worry about the self-intersection.

\begin{cor}[Random-fixed intersection formula] \label{random-fixed intersection} Let $\bs A$ be a densable sequence of permutation invariant random subsets of $\bs E$ with density $d$. Let $\bs X$ be a sequence of (fixed) subsets of $\bs E$ with density $\alpha$. If $d + \alpha \neq 1$, then the sequence of random subsets $\bs A\cap \bs X$ is densable and 
    \[\dens (\bs A\cap \bs X) = \begin{cases}
        d + \alpha - 1 & \textup{ if } d + \alpha > 1\\
        -\infty & \textup{ if } d + \alpha < 1.
        \end{cases}\]
\end{cor}\quad\\

We shall first represent the expected value and the variance of the random variable $|A_n^{(k)}\cap X_n|$ by probabilities of the type $\Pr\left(\{x_1,\dots,x_r\} \subset A_n\right)$. The following result generalize Lemma \ref{E and Var}.

\begin{lem}\label{multidim E and Var} Let $\bs E$, $\bs A$ and $\bs X$ given by Theorem \ref{multidim intersection} and let $(\bs Y_i)_{0\leq i \leq k}$ be the self-intersection partition of $\bs X$. Let $x_1,\dots,x_{2k}$ be distinct $2k$ elements of $E_n$.
\begin{enumerate}
    \item $\esp{|A_n^{(k)} \cap X_n|} = |X_n|\Pr\left(\{x_1,\dots,x_k\}\subset A_n\right)$.
    \item $\Var\left(|A_n^{(k)} \cap X_n|\right) =$
    \begin{align*}
        & |X_n|^2\Big(\Pr(\{x_1,\dots,x_{2k}\}\subset A_n) - \Pr(\{x_1,\dots,x_{k}\}\subset A_n)^2 \Big)\\
    + & \sum_{i = 1}^{k}|Y_{i,n}| \Big(\Pr(\{x_1,\dots,x_{2k-i}\}\subset A_n) - \Pr(\{x_1,\dots,x_{2k}\}\subset A_n)\Big).
    \end{align*}
\end{enumerate}
\end{lem}
\begin{proof}\quad
    \begin{enumerate}
        \item As $A_n$ is permutation invariant, the probability $\Pr(\{x_1,\dots,x_k\}\subset A_n)$ does not depend on the choice of $\{x_1,\dots,x_k\}$. So
        \[\begin{split} \mathbb{E}(|A_n^{(k)} \cap X_n|)& =  \mathbb{E}\left(\sum_{x\in X_n}\mathbbm{1}_{x\in A_n^{(k)}}\right) = \sum_{x\in X_n}\Pr\left(x\in A_n^{(k)}\right)\\
        & = |X_n|\Pr\left(\{x_1,\dots,x_k\}\subset A_n\right).
        \end{split}\]
        \item By the same reason $\Pr(\{x_1,\dots,x_r\}\subset A_n)$ does not depend on the choice of $\{x_1,\dots,x_r\}$ for all $r\in\mathbb{N}$. Note that 
        \[\Var(|A_n^{(k)} \cap X_n|) = \esp{|A_n^{(k)} \cap X_n|^2} - \esp{|A_n^{(k)} \cap X_n|}^2.\] 
        If $(x,y) \in Y_{i,n}$, then there are $2k-i$ different elements of $E_n$ in $x$ and $y$, so $\Pr\left(x,y\in A_n^{(k)}\right) = \Pr(\{x_1,\dots,x_{2k-i}\}\subset A_n)$. Hence
        \begin{align*} \esp{|A_n^{(k)} \cap X_n|^2} & = \esp{\left(\sum_{x\in X_n}\mathbbm{1}_{x\in A_n^{(k)}}\right)^2} = \sum_{x,y\in X_n}\Pr\left(x,y\in A_n^{(k)}\right)\\
        & =  \sum_{i = 0}^{k}\sum_{(x,y)\in Y_{i,n}}\Pr\left(x,y\in A_n^{(k)}\right)\\
        & =  \sum_{i = 0}^{k}|Y_{i,n}|\Pr(\{x_1,\dots,x_{2k-i}\}\subset A_n).\end{align*}
        Recall that $|Y_{0,n}| = |X_n^2| - \sum_{i=1}^k|Y_{i,n}|$. The above can be rewrite as
        \begin{align*}
        \esp{|A_n^{(k)} \cap X_n|^2} & = \left(|X_n^2| - \sum_{i = 1}^{k}|Y_{i,n}|\right) \Pr(\{x_1,\dots,x_{2k}\}\subset A_n)\\
        & \quad + \sum_{i = 1}^{k}|Y_{i,n}| \Pr(\{x_1,\dots,x_{2k-i}\}\subset A_n)\\
        & = |X_n^2|\Pr(\{x_1,\dots,x_{2k}\}\subset A_n)\\
        & \quad + \sum_{i=1}^k\Big(\Pr(\{x_1,\dots,x_{2k-i}\}\subset A_n)-\Pr(\{x_1,\dots,x_{2k}\}\subset A_n)\Big).
        \end{align*}
        Combined with $\esp{|A_n^{(k)} \cap X_n|}^2 = |X_n|^2\Pr\left(\{x_1,\dots,x_k\}\subset A_n\right)^2$, we have
        \begin{align*} \Var(|A_n^{(k)} \cap X_n|) & =  |X_n|^2\Big(\Pr(\{x_1,\dots,x_{2k}\}\subset A_n) - \Pr(\{x_1,\dots,x_{k}\}\subset A_n)^2 \Big) \\
        & + \sum_{i = 1}^{k}|Y_{i,n}| \Big(\Pr(\{x_1,\dots,x_{2k-i}\}\subset A_n) - \Pr(\{x_1,\dots,x_{2k}\}\subset A_n\Big).
        \end{align*}
    \end{enumerate}
\end{proof}

Remark that Lemma \ref{E and Var} is a special case Lemma \ref{multidim E and Var}, by taking $k = 1$ and $X_n = E_n$. Note that if $k = 1$, then $\bs X^2 = \bs Y_0 \sqcup \bs Y_1$ and there is no need to introduce condition (\ref{multidim condition}).

\subsection{The Bernoulli density model}
Let $\bs X$ be a fixed sequence of subsets of $\bs E^{(k)}$ with density $\alpha$. In this subsection, we study the intersection $\bs A^{(k)}\cap \bs X$ in the case that $\bs A$ is a sequence of Bernoulli random subsets of $\bs E$ with density $0<d<1$. Note that for any integer $r\in \mathbb{N}$ and any distinct elements $x_1,\dots,x_r$ in $E_n$, we have
\begin{align*}
    \Pr\left(\{x_1,\dots,x_r\}\subset A_n\right) & = \Pr\left(\{x_1\in A_n\},\dots, \{x_r\in A_n\}\right)\\  & = \prod_{i=1}^r\Pr(x_i\in A_n) = n^{r(d-1)}
\end{align*}
by independence of the events $\Pr(x_i\in A_n)$. Because of this equality, the proof of Theorem \ref{multidim intersection} for the Bernoulli density model is much simpler.

\begin{proof}[\textbf{Proof of Theorem \ref{multidim intersection} for Bernoulli density model}] \quad
\begin{enumerate}[(i)]
    \item Suppose that $\alpha+d<1$. To prove that $\dens(\bs A^{(k)}\cap \bs X) = -\infty$, it is enough to prove that $\Pr\left(A_n^{(k)}\cap X_n \neq \vide\right)\tendvers 0$.
    
    By Markov's inequality and Lemma \ref{multidim E and Var}
    \[\begin{split}\Pr\left(A_n^{(k)}\cap X_n \neq \vide\right) & = \Pr\left(|A_n^{(k)}\cap X_n| \geq 1\right)\\
    & \leq \esp{|A_n^{(k)}\cap X_n|} = |X_n|\Pr\left(\{x_1,\dots,x_k\}\subset A_n\right)\\
    & \leq n^{k\alpha+o(1)}n^{k(d-1)}\\
    & \leq n^{k(\alpha + d -1 )+o(1)} \tendvers 0
    \end{split}\]
    as $\alpha+d-1<0$. \qed
    
    \item Suppose that $\alpha+d>1$. To simplify the notation, denote $B_n =  A_n^{(k)}\cap X_n$ and $\bs B = \bs X\cap \bs A^{(k)}$. 
    
    We shall prove that $\dens \bs B = \alpha+d-1$. Let $\varepsilon>0$ be an arbitrary small real number. We need prove that a.a.s.
    \[n^{k(\alpha+d-1-\varepsilon)} \leq |B_n| \leq n^{k(\alpha+d-1+\varepsilon)}.\]
    
    By Lemma \ref{multidim E and Var} 
    \begin{align*} \esp{|B_n|} & = |X_n|\Pr\left(\{x_1,\dots,x_k\}\subset A_n\right) = |X_n|n^{k(d-1)} \\ 
     &= n^{k(\alpha + d -1)+o(1)}.
    \end{align*}
    For $n$ large enough
    \[n^{k(\alpha+d-1-\varepsilon)} < \frac{1}{2}n^{k(\alpha+d-1)+o(1)} < \frac{3}{2}n^{k(\alpha+d-1)+o(1)} < n^{k(\alpha+d-1+\varepsilon)}.\]

    So it is enough to prove that a.a.s.
    \[\frac{1}{2}\esp{|B_n|} < |B_n| < \frac{3}{2}\esp{|B_n|},\]
    which means that a.a.s.
    \[\left||B_n|-\esp{|B_n|}\right| < \frac{1}{2}\esp{|B_n|}.\]
   
    By Chebyshev's inequality
    \[\Pr\left(\left||B_n|-\esp{|B_n|}\right| \geq \frac{1}{2}\esp{|B_n|}\right) \leq \frac{4\Var\left(|B_n|\right)}{\esp{|B_n|}^2}.\]
    We shall prove that this quantity goes to zero when $n$ goes to infinity. By Lemma \ref{multidim E and Var}
    \begin{align*} \Var(|B_n|) & =  |X_n|^2\Big(\Pr(\{x_1,\dots,x_{2k}\}\subset A_n) - \Pr(\{x_1,\dots,x_{k}\}\subset A_n)^2 \Big) \\
    & \quad + \sum_{i = 1}^{k}|Y_{i,n}| \Big(\Pr(\{x_1,\dots,x_{2k-i}\}\subset A_n) - \Pr(\{x_1,\dots,x_{2k}\}\subset A_n\Big)\\
    & = \sum_{i = 1}^{k}|Y_{i,n}| \left(n^{(2k-i)(d-1)} - n^{2k(d-1)}\right)\\
    & \leq \sum_{i = 1}^{k}|Y_{i,n}| n^{(2k-i)(d-1)}
    \end{align*}
    Note that $n^{(2k-i)(d-1)} > n^{2k(d-1)}$ because $d<1$. By the $d$-small self-intersection condition (\ref{multidim condition}), there exists $\varepsilon >0$ such that for all $1\leq i \leq k$
    \[|Y_{i,n}|\leq n^{2k\left(\alpha+(d-1)\frac{i}{2k}\right)-\varepsilon}\]
    for $n$ large enough.\\
    Hence for $n$ large enough
    \[\Var(|B_n|)\leq kn^{2k(\alpha+d-1)-\varepsilon}.\]
    Recall that $\esp{|B_n|}^2 = n^{2k(\alpha + d -1)+o(1)}$, so
    \[\frac{4\Var\left(|B_n|\right)}{\esp{|B_n|}^2} \tendvers 0.\]
    \end{enumerate}
\end{proof}

\subsection{The uniform density model}
Note that when $\bs A$ is a sequence of Bernoulli random subsets with density $d$, we have
    \[\Pr(\{x_1,\dots,x_r\}\subset A_n) = n^{r(d-1)},\]
and consequently
    \[\Pr(\{x_1,\dots,x_k\}\subset A_n)^2 - \Pr(\{x_1,\dots,x_{2k}\}\subset A_n) = 0.\]
In order to proceed the same proof, we shall estimate these two quantities for the uniform density model.

\begin{lem}\label{uniform include} Let $\bs A$ be a sequence of uniform random subsets of $\bs E$ with density $d$. Let $0<\varepsilon<d$ be a small real number and let $k\geq 1$ be an integer. If $n\geq (1+2k)^{\frac{1}{\varepsilon}}$, then
\begin{enumerate}[(i)]
    \item For all integers $1\leq r \leq 2k$
    \[n^{r(d-1-\varepsilon)} \leq \Pr(\{x_1,\dots,x_r\}\subset A_n)\leq n^{r(d-1+\varepsilon)}.\]
    \item $0\leq\Pr(\{x_1,\dots,x_k\}\subset A_n)^2 - \Pr(\{x_1,\dots,x_{2k}\}\subset A_n) \leq n^{2k(d-1+\varepsilon)-d}$
\end{enumerate}
\end{lem}
\begin{proof} Recall that $|E_n| = n$ and that $A_n$ is uniform on all subsets of $E_n$ of cardinality $\flr{n^d}$.
\begin{enumerate}[(i)]
    \item Note that $\flr{n^d}\geq n^\varepsilon-1\geq 2k \geq r$. Among all subsets of $E_n$ of cardinality $\flr{n^d}$, there are $\binom{n-r}{\flr{n^d}-r}$ subsets that include $\{x_1,\dots,x_r\}$. So
    \[\Pr(\{x_1,\dots,x_r\}\subset A_n) = \frac{\binom{n-r}{\flr{n^d}-r}}{\binom{n}{\flr{n^d}}} = \frac{\flr{n^d}\dots(\flr{n^d}-r+1)}{n\dots(n-r-1)}.\]
    We estimate that
    \[\left(\frac{n^d-r}{n}\right)^r \leq \frac{\flr{n^d}\dots(\flr{n^d}-r+1)}{n\dots(n-r-1)} \leq \left(\frac{n^d}{n-r}\right)^r.\]
    The condition $n \geq (1+2k)^{\frac{1}{\varepsilon}} \geq (1+r)^{\frac{1}{\varepsilon}}$ implies
    \[\begin{cases}
        n \geq n^{1-\varepsilon}(1+r)\\
        n^d \geq n^{d-\varepsilon}(1+r),
    \end{cases}\]
    so
    \[\begin{cases}
        n^{1-\varepsilon} \leq n-r\\
        n^{d-\varepsilon} \leq n^d-r.
    \end{cases}\]
    Hence
    \[\left(n^{d-1-\varepsilon}\right)^r \leq \frac{\flr{n^d}\dots(\flr{n^d}-r+1)}{n\dots(n-r-1)} \leq \left(n^{d-1+\varepsilon}\right)^r.\]\qed
    
    \item By the same argument 
    \begin{align*}&\Pr(\{x_1,\dots,x_k\}\subset A_n)^2- \Pr(\{x_1,\dots,x_{2k}\}\subset A_n) \\
    = & \left(\frac{\flr{n^d}\dots(\flr{n^d}-k+1)}{n\dots(n-k-1)}\right)^2 - \frac{\flr{n^d}\dots(\flr{n^d}-2k+1)}{n\dots(n-2k-1)}\\
    = & \left(\frac{\flr{n^d}\dots(\flr{n^d}-k+1)}{n\dots(n-k-1)}\right)\left(\frac{\flr{n^d}\dots(\flr{n^d}-k+1)}{n\dots(n-k-1)} - \frac{(\flr{n^d}-k)\dots(\flr{n^d}-2k+1)}{(n-k)\dots(n-2k-1)} \right).
    \end{align*}
    This quantity is positive because 
    $\frac{\flr{n^d}-i}{n-i}\geq \frac{\flr{n^d}-i-k}{n-i-k}$
    for every $0\leq i \leq k-1$.
    
    Now we estimate that
    \begin{align*}&\Pr(\{x_1,\dots,x_k\}\subset A_n)^2- \Pr(\{x_1,\dots,x_{2k}\}\subset A_n) \\
    \leq & \left(\frac{n^d}{n-k}\right)^k\left(\frac{n^{dk}}{(n-k)^k} - \frac{(n^d-2k)^k}{(n-k)^k}\right) \\ 
    \leq & \frac{n^{dk}}{(n-k)^{2k}}\left(n^{dk} - \sum_{i=0}^k\binom{k}{i}n^{d(k-i)}(-2k)^i\right) \\
    \leq & \frac{n^{dk}}{(n-k)^{2k}}(1+2k)^kn^{d(k-1)} = \left(\frac{n^d\sqrt{1+2k}}{n-k}\right)^{2k}n^{-d}.
    \end{align*}
    As $n^\varepsilon\geq 1+2k$, we have
    \begin{align*}
        n-k & \geq n^{1-\varepsilon}(1+2k)-k\\
        & \geq n^{1-\varepsilon}(1+k)\\
        & \geq n^{1-\varepsilon}\sqrt{1+2k}.
    \end{align*}
    Hence $\Pr(\{x_1,\dots,x_k\}\subset A_n)^2 - \Pr(\{x_1,\dots,x_{2k}\}\subset A_n) \leq n^{2k(d-1+\varepsilon)-d}$. 
\end{enumerate}
\end{proof}

\begin{nota}Let $\bs X$ be a sequence of subsets of $\bs E^{(k)}$ with density $\alpha$ and let $(\bs Y_i)_{0\leq i \leq k}$ be its self-intersection partition. Denote the density difference
\[\varepsilon_0(d) = \min_{1\leq i \leq k}\left\{\alpha + (d-1)\frac{i}{2k}-\overline{\dens} \bs Y_i\right\}.\]
\end{nota}
Remark that $\bs X$ has $d$-small self-intersection if and only if $\varepsilon_0(d)>0$. In addition, for every small real number $0<\varepsilon<\frac{\varepsilon_0(d)}{10}$ there exists $n_\varepsilon\in \mathbb{N}$ such that for all $n\geq n_\varepsilon$ we have, simultaneously for all $1\leq i \leq k$,
\[|Y_{n,i}|\leq n^{2k\left(\alpha+(d-1)\frac{i}{2k}-10\varepsilon\right)} = n^{2k\alpha + (d-1)i-2k\times 10\varepsilon}.\]
By densability of $\bs X$, we can choose $n_\varepsilon$ such that at the same time 
\[n^{k(\alpha-\varepsilon)}\leq |X_n|\leq n^{k(\alpha+\varepsilon)}.\]

Combined with Lemma \ref{uniform include}, we can now estimate the expected value and the variance of $|A_n^{(k)}\cap X_n|$ for the uniform density model.
\begin{lem}\label{multidim uniform E and Var} Let $\bs A$ be a sequence of uniform random subsets of $\bs E$ with density $d$. Let $\bs X$ be a sequence of subsets of $\bs E^{(k)}$ with density $\alpha$. Let $0<\varepsilon < \min\{\frac{\varepsilon_0(d)}{10},d\}$ be a small real number. If $n\geq \max\left\{n_\varepsilon,(1+2k)^{\frac{1}{\varepsilon}}\right\}$, then
\begin{enumerate}[$(i)$]
    \item $n^{k(\alpha+d-1-2\varepsilon)} \leq \esp{|A_n^{(k)}\cap X_n|} \leq n^{k(\alpha+d-1+2\varepsilon)}$.
    \item If in addition $\alpha+d-1>2\varepsilon>0$ and $\bs X$ has $d$-small self-intersection, then\\ $\Var\left(|A_n^{(k)}\cap X_n|\right)\leq kn^{2k(\alpha+d-1-9\varepsilon)}.$
\end{enumerate}
\end{lem}
\begin{proof}\quad
\begin{enumerate}[$(i)$]
    \item By Lemma \ref{multidim E and Var}
    \[\esp{|A_n^{(k)}\cap X_n|} = |X_n|\Pr\left(\{x_1,\dots,x_k\}\subset A_n\right).\]
    So by Lemma \ref{uniform include} and $n^{k(\alpha-\varepsilon)}\leq |X_n|\leq n^{k(\alpha+\varepsilon)}$ we have
    \[n^{k(\alpha-\varepsilon)}n^{k(d-1-\varepsilon)} \leq \esp{|A_n^{(k)}\cap X_n|} \leq n^{k(\alpha+\varepsilon)}n^{k(d-1+\varepsilon)}.\]\qed
    \item By Lemma \ref{uniform include} (ii) $\Pr(\{x_1,\dots,x_{2k}\}\subset A_n) - \Pr(\{x_1,\dots,x_{k}\}\subset A_n)^2 \leq 0$. Apply Lemma \ref{multidim E and Var}, eliminate negative parts.
    \begin{align*} \Var\left(|A_n^{(k)}\cap X_n|\right) & =  |X_n|^2\Big(\Pr(\{x_1,\dots,x_{2k}\}\subset A_n) - \Pr(\{x_1,\dots,x_{k}\}\subset A_n)^2\Big) \\
    & + \sum_{i=1}^{k}|Y_{i,n}|\Big(\Pr(\{x_1,\dots,x_{2k-i}\}\subset A_n) - \Pr(\{x_1,\dots,x_{2k}\}\subset A_n)\Big)\\
    & \leq \sum_{i=1}^{k}|Y_{i,n}|\Pr(\{x_1,\dots,x_{2k-i}\}\subset A_n).
    \end{align*}
    By Lemma \ref{uniform include} (i) and $|Y_{i,n}|\leq n^{2k\alpha+i(d-1)+2k\times10\varepsilon}$
    \begin{align*}
    \Var\left(|A_n^{(k)}\cap X_n|\right) & \leq \sum_{i=1}^{k}n^{2k\alpha+i(d-1)-2k\times10\varepsilon} n^{(2k-i)(d-1+\varepsilon)}\\
    & \leq kn^{2k(\alpha+d-1-9\varepsilon)}.
    \end{align*}
\end{enumerate}
\end{proof}

\begin{proof}[\textbf{Proof of Theorem \ref{multidim intersection} for uniform density model}]\quad
\begin{enumerate}[$(i)$]
    \item Suppose that $\alpha+d<1$. We shall prove that $\Pr\left(A_n^{(k)}\cap X_n \neq \vide\right)\tendvers 0$.
    
    Let $\varepsilon>0$ such that
    \[\varepsilon<\min\left\{\frac{1-d-\alpha}{2}, \frac{\varepsilon_0(d)}{10}, d\right\}.\]
    
    By Markov's inequality and Lemma \ref{multidim uniform E and Var}. If $n \geq \max\{n_\varepsilon,(1+2k)^{\frac{1}{\varepsilon}}\}$, then
    \[\begin{split}\Pr\left(A_n^{(k)}\cap X_n \neq \vide\right) & = \Pr\left(|A_n^{(k)}\cap X_n| \geq 1\right)\\
    & \leq \esp{|A_n^{(k)}\cap X_n|}\\
    & \leq n^{k(\alpha+d-1+2\varepsilon)}\tendvers 0.
    \end{split}\]\qed
    
    \item Suppose that $\alpha+d>1$. Denote $B_n =  A_n^{(k)}\cap X_n$.
    
    Let $\varepsilon>0$ be an arbitrary small number, with     \[\varepsilon<\min\left\{\frac{\alpha+d-1}{3}, \frac{\varepsilon_0(d)}{10}, d\right\}.\]

    We shall prove that a.a.s.
    \[n^{k(\alpha+d-1-3\varepsilon)} \leq |B_n| \leq n^{k(\alpha+d-1+3\varepsilon)}.\]
    
    By Lemma \ref{multidim uniform E and Var}, if $n \geq \max\{n_\varepsilon,(1+2k)^{\frac{1}{\varepsilon}}\}$, then
    \begin{align*} n^{k(\alpha+d-1-2\varepsilon)} \leq \esp{|B_n|} \leq n^{k(\alpha+d-1+2\varepsilon)}.
    \end{align*}
    In addition, if $n\geq 2^{\frac{1}{k\varepsilon}}$, we have
    \[n^{k(\alpha+d-1-3\varepsilon)} \leq \frac{1}{2}n^{k(\alpha+d-1-2\varepsilon)} \leq \frac{1}{2}\esp{|B_n|}\]
    and 
    \[\frac{3}{2}\esp{|B_n|} \leq \frac{3}{2}n^{k(\alpha+d-1+2\varepsilon)} \leq n^{k(\alpha+d-1+3\varepsilon)}.\]

    So it is enough to prove that a.a.s.
    \[\left||B_n|-\esp{|B_n|}\right| \leq \frac{1}{2}\esp{|B_n|}.\]
   
    By Chebyshev's inequality
    \[\Pr\left(\left||B_n|-\esp{|B_n|}\right| > \frac{1}{2}\esp{|B_n|}\right) \leq \frac{4\Var\left(|B_n|\right)}{\esp{|B_n|}^2}.\]

    Combined with Lemma \ref{multidim uniform E and Var}, if $n\geq \max\left\{ n_\varepsilon, (1+2k)^{\frac{1}{\varepsilon}}, 2^{\frac{1}{k\varepsilon}}\right\}$, then
    \[\frac{4\Var\left(|B_n|\right)}{\esp{|B_n|}^2}\leq \frac{4kn^{2k(\alpha+d-1-9\varepsilon)}}{n^{2k(\alpha+d-1-2\varepsilon)}} \leq \frac{4k}{n^{14k\varepsilon}} \tendvers 0.\]
\end{enumerate}
\end{proof}

\subsection{The general model (densable and permutation invariant)}
Let $\bs X$ be a sequence of subsets of $\bs E^{(k)}$ having the $d$-small intersection condition. Recall that $\varepsilon_0(d) = \min_{1\leq i \leq k}\left\{\alpha + (d-1)\frac{i}{2k}-\overline{\dens} \bs Y_i\right\} > 0$. Note that if $d'<d$ then $\varepsilon(d')<\varepsilon(d)$.

In order to apply lemma \ref{multidim uniform E and Var} in a small interval $[d-\varepsilon,d+\varepsilon]$, we choose $0<\varepsilon<\min\left\{ \frac{\varepsilon_0(d)}{20}, \frac{d}{2}\right\}$ so that $\varepsilon < \min\left\{ \frac{\varepsilon_0(d-\varepsilon)}{10}, d-\varepsilon \right\} \leq \min\left\{\frac{\varepsilon_0(d')}{10},d'\right\}$ for every $d'\in [d-\varepsilon,d+\varepsilon]$.

By the definition of $\varepsilon_0$ and the densability of $\bs X$, we choose again $n_\varepsilon\in \mathbb{N}$ such that for all $n\geq n_\varepsilon$
\[|Y_{n,i}|\leq n^{2k\alpha + (d-1)i-2k\times 20\varepsilon} \leq n^{2k\alpha + (d'-1)i-2k\times 10\varepsilon}\quad \forall 1\leq i \leq k\]
and
\[n^{k(\alpha-\varepsilon)}\leq |X_n|\leq n^{k(\alpha+\varepsilon)}.\]

\begin{lem}\label{multidim uniformly uniform E and Var} 
Let $0 < \varepsilon < \min\left\{ \frac{\varepsilon_0(d)}{20}, \frac{d}{2} \right\}$ be a small real number. Let $\bs A$ be a sequence of uniform random subsets of $\bs E$ with density $d'\in [d-\varepsilon,d+\varepsilon]$. Let $\bs X$ be a sequence of subsets of $\bs E^{(k)}$ with density $\alpha$. If $n\geq \max\left\{n_\varepsilon,(1+2k)^{\frac{1}{\varepsilon}}\right\}$, then
\begin{enumerate}[(i)]
    \item $n^{k(\alpha + d -1 - 3\varepsilon)} \leq \esp{|A_n^{(k)}\cap  X_n|} \leq n^{k(\alpha + d -1 + 3\varepsilon)}$.
    \item If in addition $\alpha+d-1>3\varepsilon>0$ and $\bs X$ has $d$-small self-intersection, then\\ $\Var(|A_n^{(k)}\cap X_n|)\leq kn^{2k(\alpha+d-1-8\varepsilon)}$.
\end{enumerate}
\end{lem}
\begin{proof}\quad
    \begin{enumerate} [(i)]
        \item Recall from the above discussion that $\varepsilon < \min\left\{\frac{\varepsilon_0(d')}{10},d'\right\}$. By Lemma \ref{multidim uniform E and Var}
        \[n^{k(\alpha + d' -1 - 2\varepsilon)} \leq \esp{|A_n^{(k)}\cap  X_n|} \leq n^{k(\alpha + d' -1 + 2\varepsilon)}.\]
        We then have the inequality by $d-\varepsilon\leq d'\leq d+\varepsilon$.
        \item  Because $\varepsilon_0(d')>0$, $\bs X$ has $d'$-small self-intersection. By Lemma \ref{multidim uniform E and Var} and the fact that $d'\leq d+\varepsilon$,
        \begin{align*}\Var(|A_n^{(k)}\cap  X_n|) \leq kn^{2k(\alpha+d'-1-9\varepsilon)} \leq kn^{2k(\alpha+d-1-8\varepsilon)}.
        \end{align*}
    \end{enumerate}
\end{proof}

\begin{lem}[Concentration lemma]\label{multidim uniform concentration} Let $\varepsilon>0$ be an arbitrary small real number. Let $\bs A$ and $\bs X$ given as the previous lemma, with $\alpha+d-1>4\varepsilon>0$ and $\bs X$ having $d$-small self-intersection. If $\varepsilon < \min\{\frac{\varepsilon_0}{10},\frac{d}{2}\}$ and $n\geq \max\left\{n_\varepsilon,(1+2k)^{\frac{1}{\varepsilon}}\right\}$, then
\[\Pr\left(n^{k(\alpha + d -1 - 4\varepsilon)} \leq |A_n^{(k)}\cap  X_n| \leq n^{k(\alpha + d -1 + 4\varepsilon)}\right) > 1 - kn^{-10k\varepsilon}.\]
\end{lem}
\begin{proof}
    Denote $B_n = A_n^{(k)}\cap  X_n$. By \ref{multidim uniformly uniform E and Var}$(i)$ and $n^{k\varepsilon} \geq 2$, we have
    \begin{align*}
            n^{k(\alpha+d-1-4\varepsilon)} \leq  \frac{1}{2}n^{k(\alpha+d-1-3\varepsilon)} \leq \frac{1}{2}\esp{|B_n|}
    \end{align*}
    and
    \begin{align*}
            \frac{3}{2}\esp{|B_n|} \leq  \frac{3}{2}n^{k(\alpha+d-1+3\varepsilon)} \leq n^{k(\alpha+d-1+4\varepsilon)}.
    \end{align*}
    By Chebyshev's inequality
    \begin{align*}
            & \Pr\left(n^{k(\alpha+d-1-4\varepsilon)} \leq |B_n| \leq n^{k(\alpha + d -1 + 4\varepsilon)}\right)\\
            \geq & \Pr\left(\Big||B_n|-\esp{|B_n|}\Big| \leq\frac{1}{2}\esp{|B_n|}\right)\\
            \geq & 1-\frac{4\Var(|B_n|)}{\esp{|B_n|}^2}.
    \end{align*}
    Again by Lemma \ref{multidim uniformly uniform E and Var}
    \begin{align*}
        \frac{4\Var(|B_n|)}{\esp{|B_n|}^2} \leq & \frac{kn^{2k(\alpha+d-1-8\varepsilon)}}{n^{2k(\alpha+d-1-3\varepsilon)}}\\
        \leq & kn^{-10k\varepsilon}.
    \end{align*}
\end{proof}

\begin{proof}[\textbf{Proof of the Theorem \ref{multidim intersection}}] \quad

Let $\varepsilon >0$ be an arbitrary small number, with $\varepsilon < \min\left\{\frac{d}{2},\frac{\varepsilon_0(d)}{20}\right\}$ as given in Lemma \ref{multidim uniformly uniform E and Var}. Denote $Q_n = \{n^{d-\varepsilon}\leq |A_n| \leq n^{d+\varepsilon}\}$ and 
    \[\mathbb{N}_{\bs A,\varepsilon,n} := \left\{\ell\in\mathbb{N} \;|\; n^{d-\varepsilon}\leq \ell \leq n^{d+\varepsilon}\textup{ and } \Pr(|A_n| = \ell)>0\right\}.\]
By densability of $\bs A$ we have $\Pr(Q_n)\tendvers 1$. Denote by $\Pr_{Q_n}:=\PrCond{\cdot}{Q_n}$ the probability measure under the condition $Q_n$. Define similarly $\mathbb{E}_{Q_n}$ and $\Var_{Q_n}$.

In order to prove that some sequence of properties $(R_n)$ is a.a.s. true, by the inequality
\[\Pr(\overline{R_n})\leq \Pr(Q_n)\Pr_{Q_n}({\overline{R_n}}) + \Pr(\overline{Q_n}),\]
it is enough to prove that $\Pr_{Q_n}({\overline{R_n}})\tendvers 0$.

\begin{enumerate}[(i)]
    \item Suppose that $\alpha+d<1$. Assume in addition that $\varepsilon < \frac{1-d-\alpha}{3}$.
    
    We shall prove that \[\Pr_{Q_n}(A_n^{(k)}\cap X_n \neq \vide) = \Pr_{Q_n}(|A_n^{(k)}\cap X_n| \geq 1)\tendvers 0.\]
    
     By the formula of total probability and Markov's inequality:
    \begin{align*} \Pr_{Q_n}\left(|A_n^{(k)}\cap X_n|\geq 1 \right)
    & \leq \sum_{l\in \mathbb{N}_{A,\varepsilon,n}}\Pr_{Q_n}(A_n = l) \Pr\Cond{|A_n^{(k)}\cap X_n|\geq 1}{|A_n| = l}\\
    & \leq \sum_{l\in \mathbb{N}_{A,\varepsilon,n}}\Pr_{Q_n}(A_n = l) \mathbb{E}\Cond{|A_n^{(k)}\cap X_n|}{|A_n| = l}.
    \end{align*}
    
    By a change of variable $l = n^{d'}$ with $d-\varepsilon\leq d'\leq d+\varepsilon$, apply Lemma \ref{multidim uniformly uniform E and Var}
    \begin{align*} \Pr_{Q_n}\left(|A_n^{(k)}\cap X_n|\geq 1 \right)
    & \leq \sum_{l\in \mathbb{N}_{A,\varepsilon,n}}\Pr_{Q_n}(A_n=l=n^{d'}) n^{\alpha+d-1+3\varepsilon}\\
    & \leq n^{\alpha+d-1+3\varepsilon}\tendvers 0.
    \end{align*}\qed

    \item Suppose that $\alpha+d>1$. Assume in addition that $\varepsilon < \frac{\alpha+d-1}{4}$, so that we can apply Lemma \ref{multidim uniform concentration}.
    
    We shall prove that
    \[\Pr_{Q_n}\left(n^{k(\alpha + d -1 - 4\varepsilon)} \leq |A_n^{(k)}\cap X_n| \leq n^{k(\alpha + d -1 + 4\varepsilon)}\right)\tendvers 1.\]
    By the formula of total probability, Lemma \ref{multidim uniform concentration} and a change of variables $l=n^{d'}$:
    \begin{align*}
        & \Pr_{Q_n}\left(n^{k(\alpha + d -1 - 4\varepsilon)} \leq |A_n^{(k)}\cap X_n| \leq n^{k(\alpha + d -1 + 4\varepsilon)}\right) \\ 
        =\, & \sum_{l\in \mathbb{N}_{A,\varepsilon,n}}\Pr_{Q_n}(A_n = l) \Pr\Cond{n^{k(\alpha + d -1 - 4\varepsilon)} \leq |A_n^{(k)}\cap X_n| \leq n^{k(\alpha + d -1 + 4\varepsilon)}}{|A_n| = l}\\
        \geq\, & \sum_{l\in \mathbb{N}_{A,\varepsilon,n}}\Pr_{Q_n}(A_n = l = n^{d'}) \left(1 - kn^{-10k\varepsilon}\right)\\
        \geq\, & 1 - kn^{-10k\varepsilon}\tendvers 1.
    \end{align*}
\end{enumerate}

\end{proof}

\section{Applications to group theory}

Fix an alphabet $X = \{x_1,\dots, x_m\}$ as generators of groups. Let $B_\ell$ be the set of cyclically reduced words of length at most $\ell$ on $X^\pm$. Recall that $|B_\ell| = (2m-1)^{\ell+O(1)}$.

We are interested in asymptotic behaviors, when $\ell$ goes to infinity, of group presentations $\langle X|R_\ell \rangle$ where $R_\ell$ is a random subset of $B_\ell$.\\

\begin{defi}[Random groups with density] Let $d\in ]0,1]$. Let $\bs R = (R_\ell)$ be a densable sequence of permutation invariant random subsets with density $d$ of the sequence $\bs B = (B_\ell)$.

Denote $G_\ell = G_\ell(m,d)$ the random presentation defined by $\langle X | R_\ell \rangle$. The sequence $\bs G = \bs G(m,d) = (G_\ell(m,d))_{\ell \in \mathbb{N}}$ is called a sequence of random groups with density $d$.
\end{defi}

For example, if $d = 1$, then $G_\ell(m,1)$ is isomorphic to the trivial group.

A sequence of events $\bs Q = (Q_\ell)$ described by $\bs G(m,d)$ is \textit{asymptotically almost surely} satisfied if $\Pr(Q_\ell)\xrightarrow[m\to\infty]{}1$. We denote briefly a.a.s. $Q_\ell$.

\subsection{Phase transition at density $1/2$}

\begin{thm}[Gromov, phase transition at density 1/2]\label{hyp} Let $(G_\ell(m,d))$ be a sequence of random groups with density $d$.
\begin{enumerate} [$(i)$]
    \item If $d > 1/2$, then a.a.s $G_\ell(m,d)$ is isomorphic to the trivial group.
    \item If $d < 1/2$, then a.a.s $G_\ell(m,d)$ is a hyperbolic group.
\end{enumerate}
\end{thm}

In \cite{Oll04} 2.1 (or \cite{Oll05} I.2.b), Ollivier proved the first assertion by probabilistic pigeon-hole principle. We give a proof here by the intersection formulae (Theorem \ref{intersection} and Corollary \ref{random-fixed intersection}).

\begin{proof}[Proof of Theorem \ref{hyp} $(i)$] Let $x\in X$. Let $A_\ell$ be the set of cyclically reduced words that does not start or end by $x$, of lengths at most $\ell-1$ (so that $xA_\ell\subset B_\ell$). It is easy to check that the sequences $(A_\ell)$ and $(xA_\ell)$ are sequences of fixed subsets of $\bs B = (B_\ell)$ of density $1$. By the random-fixed intersection formula (Corollary \ref{random-fixed intersection}), the sequences $(x(R_\ell\cap A_\ell))$ and $(R_\ell\cap xA_\ell)$ are sequences of permutation invariant random subsets of $(xA_\ell)$ of density $d$.

By the intersection formula (Theorem \ref{intersection}), their intersection $(xR_\ell \cap R_\ell \cap xA_\ell)$ is a sequence of permutation invariant random subsets of $(xA_\ell)$ of density $(2d-1)>0$, which is a.a.s. not empty. Thus, a.a.s. there exists a word $w\in A_\ell$ such that $w\in R_\ell$ and $xw\in R_\ell$, so a.a.s. $x=1$ in $G_\ell$ by canceling $w$.

The argument above works for any generator $x\in X$. By intersecting a finite number of a.a.s. satisfied events, a.a.s. all generators $x\in X$ are trivial in $G_\ell$. Hence a.a.s. $G_\ell$ is isomorphic to the trivial group.
\end{proof}\quad

The proof of Theorem \ref{hyp} (ii) needs \textit{van Kampen diagrams} and will not be treated here. See \cite{Gro93} 9.B. for the original idea by Gromov, and \cite{Oll04} 2.2 or \cite{Oll05} Section V for a detailed proof by Ollivier.

\subsection{Phase transition at density $\lambda/2$}

\begin{thm}\label{lambda small cancellation} Let $\bs G(m,d) = (G_\ell(m,d))$ be a sequence of random groups with density $d$. Let $\lambda\in ]0,1[$.
\begin{enumerate}
    \item If $d<\lambda/2$, then a.a.s. $G_\ell(m,d)$ satisfies $C'(\lambda)$.
    \item If $d>\lambda/2$, then a.a.s. $G_\ell(m,d)$ \textbf{does not} satisfy $C'(\lambda)$.
\end{enumerate}
\end{thm}

\begin{proof} \quad
\begin{enumerate}
    \item Recall that (\cite{LS77} p.240) a \textit{piece} with respect to a set of relators is a cyclic sub-word that appears at least twice. There are two cases to verify.
    \begin{enumerate}
        \item  Let $A_\ell$ be the set of cyclically reduced words of length at most $\ell$ having a piece appearing twice on itself (figure 1) that is longer than $\lambda$ times itself. We shall prove that a.a.s. the intersection $A_\ell \cap R_\ell$ is empty.
        \begin{center}
            \begin{tikzpicture}
            \fill[gray!20] circle (1) ;
            \draw[very thick] circle (1);
            \node at (0,0) {\large{$r$}};
            \draw [line width = 3] (1,0) arc (0:60:1);
            \draw [line width = 3] (150:1) arc (150:210:1);
            \node at (0,-1.4) {figure 1};
            \end{tikzpicture}
        \end{center}
        We estimate first the number of relators of length $t\leq \ell$ with a piece of length $s\geq\lambda t$. There are $2t$ ways (including orientations) to choose the first position of the piece, and $2t-s$ ways the choose the second position (note that because $r$ is reduced, it can not overlay the first one if they are with opposite orientations). For each way of positioning we can determine freely $t-s$ letters, each with $(2m-1)$ choices, except for the first letter and the last letter having respectively $2m$ and $2m-2$ or $2m-1$ choices. So this number is $2t(2t-s)C(m)(2m-1)^{t-s}$ where $C(m)$ is a real number that depends only on $m$. Hence
        \[|A_\ell| = \sum_{t=1}^\ell \sum_{s=\flr{\lambda t}}^t 2t(2t-s)C(m)(2m-1)^{t-s} = (2m-1)^{(1-\lambda)\ell+o(\ell)},\]
        
        which means that $(A_\ell)$ is a sequence of fixed subsets of $(B_\ell)$ with density $1-\lambda$. By the intersection formula (Corollary \ref{random-fixed intersection}), because $1-\lambda+d<1$, we have a.a.s.
        \[A_\ell \cap R_\ell = \vide.\]
        \item Let $X_\ell$ be the set of distinct pairs of relators $r_1,r_2$ in $B_\ell$ having a piece (figure 2) longer than $\lambda\min\{|r_1|,|r_2|\}$. It is a fixed subset of $B_\ell^{(2)}$. We shall prove that a.a.s the intersection $X_\ell \cap R_\ell^{(2)}$ is empty.
        \begin{center}
        \begin{tikzpicture}
        \fill[gray!20] (30:1) circle (1) ;
        \fill[gray!20] (150:1) circle (1) ;
        \draw[very thick] (0,0) -- (0,1);
        \draw[very thick] (0,0) arc (-150:150:1);
        \draw[very thick] (0,1) arc (30:330:1);
        \node at (150:1) {\large{$r_1$}};
        \node at (30:1) {\large{$r_2$}};
        \node at (0,-0.9) {figure 2};
        \end{tikzpicture}
        \end{center}
        There are $4\ell^2$ possible positions for pieces, $(2m-1)^{\ell+o(\ell)}$ choices for $r_1$ and $(2m-1)^{\ell-\lambda\ell+o(\ell)}$ choices for $r_2$. So 
        \[|X_\ell| = (2m-1)^{(2-\lambda)\ell+o(\ell)},\]
        which means that $(X_\ell)$ is a sequence of fixed subsets of $(B_\ell^{(2)})$ with density $1-\frac{\lambda}{2}$. By the multi-dimension intersection formula (Theorem \ref{multidim intersection} (i)), because $1-\frac{\lambda}{2}+d<1$, we have a.a.s.
        \[X_\ell \cap R_\ell^{(2)} = \vide.\]
    \end{enumerate}
    \item Take the sequence of sets $\bs X = (X_\ell)$ constructed in 1(b). We shall prove that a.a.s. the intersection $X_\ell \cap R_\ell^{(2)}$ is \textit{not} empty. We have already
    \[\dens \bs X + \dens \bs R^{(2)} > 1.\]
    
    To apply Theorem \ref{multidim intersection}(ii), we need to calculate the size of the self-intersection
    \[Y_{1,\ell} = \{(x_1,x_2)\in X_\ell^2 \, | \, |x_1 \cap x_2|=1\}.\]
    Take $x_1 = (r_1,r_2)$ and $x_2=(r_1,r_3)$ being with $r_1,r_2,r_3$ three different relators in $B_\ell$. There are $(2m-1)^{\ell + o(\ell)}$ choices for $r_1$, $(2m-1)^{\ell-\lambda\ell + o(\ell)}$ choices for $r_2$ and $(2m-1)^{\ell-\lambda\ell + o(\ell)}$ choices for $r_3$. Other cases of $(x_1,x_2)$ are symmetric, so
    \[|Y_{1,\ell}| = (2m-1)^{3\ell-2\lambda\ell + o(\ell)}.\]
    Hence the density of $\bs Y_1 = (Y_{1,\ell})$ is $\frac{3-2\lambda}{4}$ in $(B_\ell^{(2)})^2$. As $d>0$, we have $\frac{3-2\lambda}{4} < 1-\frac{\lambda}{2}+\frac{1}{4}(d-1)$, which implies
    \[\dens \bs Y_1 < \dens\bs X + (d-1)\frac{1}{2\times 2}.\]
    Thus we have the $d$-small self intersection condition (definition \ref{multidim condition}). By the multi-dimensional intersection formula, a.a.s.
    \[X_\ell \cap R_\ell^{(2)} \neq \vide.\]
\end{enumerate}
\end{proof}

\subsection{Every $(m-1)$-generated subgroup is free}
Fix the set of $m$ generators $X = \{x_1,\dots, x_m\}$. Recall that $B_\ell$ is the set of $(2m-1)^{\ell+o(\ell)}$ cyclically reduced words on $X^\pm = \{x_1^\pm,\dots, x_m^\pm\}$ of length at most $\ell$. The few relator model of random groups is constructed as follows: fix a number $k\in\mathbb{N}$ and let
\[G_\ell =\langle x_1,\dots, x_m | r_1,\dots,r_k \rangle\]
where $R_\ell = \{r_1,\dots,r_k\}$ is a random subset of $B_\ell$ given by the uniform probability on all subsets of $B_\ell$ with cardinality $k$.

The sequence $(G_\ell)_{\ell\in\mathbb{N}}$ is called a \textit{sequence of random groups with $k$ relators}. As $k$ is independent of $\ell$, the sequence $(G_\ell)$ is a sequence of random groups with density $d=0$. By Proposition \ref{lambda small cancellation}, a.a.s. $G_\ell$ satisfies $C'(\lambda)$ for arbitrary small $\lambda>0$.

In \cite{AO96}, Arzhantseva and Ol'shanskii proved the following result:
\begin{thm}[Arzhantseva-Ol'shanskii, \cite{AO96} Theorem 1]
Let $(G_\ell)$ be a sequence of random groups with $k$ relators. Then a.a.s. every $(m-1)$-generated subgroup of $G_\ell$ is free.
\end{thm}

Combining the intersection formula and their arguments, we prove:
\begin{thm} \label{AO with density}
Let $(G_\ell(m,d))$ be a sequence of random groups with density $0\leq d<\frac{1}{120m^2\ln(2m)}$. Then a.a.s. every $(m-1)$-generated subgroup of $G_\ell(m,d)$ is free.
\end{thm}\quad

Let us recall the definition of "$\mu$-readable words" in \cite{AO96}.

\begin{defi}[\cite{AO96} \S2] Let $0<\mu\leq 1$. A cyclically reduced word $w$ of length $\ell$ on $X^\pm$ is $\mu$-readable if there exists a graph $\Gamma$ marked by $X^\pm$ with the following properties :
\begin{enumerate}[(a)]
    \item the number of edges of $\Gamma$ is less than $\mu\ell$;
    \item the rank of $\Gamma$ is at most $m-1$;
    \item the word $w$ can be read along some path of $\Gamma$.
\end{enumerate}
\end{defi}
Note that the condition $(b)$ is essential, because every word on $X^\pm$ can be read along the wedge of $m$ circles of length $1$ marked by $x_1,\dots, x_m$ respectively.\\

Let $M_\ell^\mu$ be the set of words $r\in B_\ell$ having a cyclic sub-word $w<r$ such that $|w|\geq \frac{1}{2}|r|$ and $w$ is $\mu$-readable. We admit the following two lemmas in \cite{AO96}.
\begin{lem}[\cite{AO96} Lemma 4] \label{density mu readable} If $\mu < \log_{2m}\left(1+\frac{1}{4m-4}\right)$, then there exists a constant $C(\mu,m)$ such that
\[|M_\ell^\mu| \leq C(\mu,m)\ell^2\left(2m-\frac{5}{4}\right)^\ell.\] \qed
\end{lem}

Recall that $|B_\ell| = (2m-1)^{\ell+O(1)}$, so $(M_\ell^\mu)$ is a densable sequence of subsets of $(B_\ell)$ with density $\log_{2m-1}\left(2m-\frac{5}{4}\right)$.

\begin{lem}[\cite{AO96} \S4]\label{lambda-mu condition} Let $G = \langle X|R\rangle$ be a group presentation where $X = \{x_1,\dots,x_m\}$ and $R$ is a subset of $B_\ell$. Suppose that \[\mu < \log_{2m}\left(1+\frac{1}{4m-4}\right) \textup{\quad and\quad} \lambda\leq \frac{\mu}{15m+3\mu}.\] 

If $R$ does not intersect $M_\ell^\mu$, has no true powers, and satisfies $C'(\lambda)$, then every $(m-1)$-generated subgroup of $G$ is free. \qed
\end{lem}\quad

\begin{proof}[\textbf{Proof of Theorem \ref{AO with density}}] We look for a density $d(m)\leq 1/2$ such that for any $d<d(m)$ a.a.s. the random group $G_\ell(m,d) = \langle X| R_\ell \rangle$ satisfies the conditions of Lemma \ref{lambda-mu condition} with $\mu = \log_{2m} \left(1+\frac{1}{4m-4}\right)-\varepsilon$ and $\lambda = \frac{\mu}{15m+3\mu}$ with an arbitrary small $\varepsilon>0$. 

The set of true powers in $B_\ell$ is with density $1/2$. By the intersection formula (Corollary \ref{random-fixed intersection}), because $d<1/2$, a.a.s. $R_\ell$ has no true powers. By Lemma \ref{density mu readable} and the intersection formula, we need $d(m) < 1 - \dens(M_\ell^\mu) <1 - \log_{2m-1}\left(2m-\frac{5}{4}\right)$ so that a.a.s. $R_\ell$ does not intersect $M_\ell^\mu$ by the intersection formula.  

At the end we need a.a.s. $R_\ell$ satisfies $C'(\lambda)$ with
\[\lambda = \frac{\log_{2m}\left(1+\frac{1}{4m-4}\right)-\varepsilon}{15m+3\log_{2m}\left(1+\frac{1}{4m-4}\right)-3\varepsilon}.\]
By Theorem \ref{lambda small cancellation}, we need $d(m)<\lambda/2$. Note that this inequality implies the previous one. For $\varepsilon$ small enough we have $\lambda > \frac{1}{60m^2\ln(2m)}$. It is enough to take \[d(m)=\frac{1}{120m^2\ln(2m)}.\]
\end{proof}


\end{document}